\documentclass[a4paper]{amsart}
\usepackage{amsmath}
\usepackage{amssymb}
\usepackage[arrow, matrix, curve]{xy}
\usepackage{xypic, calc}
\usepackage{hyperref}

\usepackage{tikz}

\usetikzlibrary{matrix,calc,arrows}

\def\to{\rightarrow}

\def\Bar{\mathcal B}
\def\bBar{\bar\Bar}
\def\coBar{\mathcal B^\vee}
\def\OOmega{\mathbf \Omega}

\def\AA{\mathbb A}
\def\AAop{\AA^{op}}
\def\RR{\mathbb R}
\def\RRop{\RR^{op}}
\def\Hom{\text{Hom}}
\def\prodc{\prod^{\quad \ c}}

\def\CC{\mathbb C}

\def\Ch{Ch}

\def\isomap{\xrightarrow{\raisebox{-0.7ex}[0ex][0ex]{$\sim$}}}

\def\mono{\rightarrowtail}

\DeclareMathOperator*{\sbigotimes}{\text{\raisebox{0.25ex}{\scalebox{0.8}{$\bigotimes$}}}}

\DeclareMathOperator*{\sbigoplus}{\text{\raisebox{0.25ex}{\scalebox{0.8}{$\bigoplus$}}}}

\newcommand{\acc}{\text{\raisebox{\depth}{\textexclamdown}}}

\newtheorem{theorem}{Theorem}[section]
\newtheorem{lemma}[theorem]{Lemma}
\newtheorem{prop}[theorem]{Proposition}

\theoremstyle{definition}

\newtheorem{example}[theorem]{Example}

\theoremstyle{remark}
\newtheorem{remark}[theorem]{Remark}

\numberwithin{equation}{section}

\begin{document}

  \title{Koszul duality for algebras over infinity-operads}

\author[E. Hoffbeck]{Eric Hoffbeck}
\address{Université Sorbonne Paris Nord, LAGA, CNRS, UMR 7539, F-93430, Villetaneuse, France}
\email{hoffbeck@math.univ-paris13.fr}

\author[I. Moerdijk]{Ieke Moerdijk}
\address{Department of Mathematics, Utrecht University, PO BOX 80.010, 3508 TA Utrecht, The Netherlands}
\email{i.moerdijk@uu.nl}
 


\thanks{}
%
\begin{abstract}
In this paper, we introduce a new notion of algebra over a linear $\infty$-operad and a corresponding notion of coalgebra over an $\infty$-cooperad. We next extend the Koszul duality between linear $\infty$-operads and linear $\infty$-cooperads from~\cite{HM21} to their categories of algebras and coalgebras. This duality theorem specialises to the known duality in the case of algebras over classical (non-infinity) operads, but our proof is different. In fact, it is based on a much more general duality between presheaves and copresheaves on a category of trees. The latter duality is a priori independent of the (co)algebra structures, but we show that it can be lifted to (co)presheaves supporting such a structure. Based on this duality, we define the homology of an algebra over an $\infty$-operad, and prove that it can be described in terms of the homology of the same category of trees with coefficients in a presheaf. 
\end{abstract}

\maketitle
 

%

\section*{Introduction}

The goal of this paper is to construct the bar complex $\Bar(P,A)$ of an algebra $A$ over an operad $P$, prove a duality theorem about it, and construct the Andr\'e-Quillen homology groups of $A$ as the derived functor of the indecomposables.
When phrased like this, these constructions and results may sound familiar to the reader. (We will give some references later.)
However, we will work in the general context of linear $\infty$-operads, and it is a priori not even clear what an algebra over such an operad should be. To elaborate on this last point right away, recall from~\cite{HM21} that a linear $\infty$-operad can be defined as a presheaf $X$ on a category $\AA$ of trees with values in chain complexes, equipped with additional structure maps $X (S \circ_e T) \to X(S) \otimes X(T)$ for any grafting $S \circ_e T$ of a tree $T$ onto a leaf $e$ of another tree $S$, satisfying some standard associativity conditions. 
These structure maps are then required to be quasi-isomorphisms. One can define $\infty$-cooperads in a dual way, as copresheaves with additional structure maps. In \emph{loc. cit.}, we proved the following.

\medskip

\begin{theoremintro}
 \begin{enumerate}
  \item There is a ``cobar-bar'' pair of adjoint functors $\coBar: CoPsh(\AA) \to Psh(\AA)$ and $\Bar: Psh(\AA) \to CoPsh(\AA)$ between the categories of presheaves and copresheaves on $\AA$ with values in chain complexes.
  \item These functors are mutually inverse up to quasi-isomorphism.
  \item The adjunction lifts to the categories of $\infty$-operads and $\infty$-cooperads.
 \end{enumerate}
\end{theoremintro}

\medskip
We will give a more precise formulation of this theorem and the constructions involved in Section~\ref{Section:Operads} so that this paper can be read independently from~\cite{HM21}; moreover our sign conventions here are slightly different.

In this paper, we will introduce an extension of $\AA$ to a larger category $\RR$ of trees, and first prove the analogue of Part~(1) of the theorem above. Since this category $\RR$ does not enjoy the same finiteness properties that $\AA$ does, however, we will have to restrict ourselves to \emph{conilpotent} presheaves (cf. Section~\ref{Conilp} and Proposition~\ref{adjonctionAlg} below).

\medskip

\begin{theoremintro}
There are adjoint functors 
$$\coBar: CoPsh^c(\RR) \rightleftarrows Psh(\RR) : \Bar$$
between the categories of conilpotent copresheaves and that of presheaves on $\RR$ with values in chain complexes.
The unit and counit of this adjunction are quasi-isomorphisms.
\end{theoremintro}

\medskip

The motivation for the definition of the category $\RR$ comes from the fact that for an $\infty$-operad $X$ (a presheaf on $\AA$), there is a natural way of defining an $\infty$-algebra over $X$ as a presheaf $M$ on $\RR$ equipped with certain structure maps relating $M$ to $X$, a bit like $M(S \circ_e T) \to X(S) \otimes M(T)$ for a grafting as above ; see Section~\ref{defalgebras} for a precise definition.
With a similar dual definition of a conilpotent $\infty$-coalgebra over an $\infty$-cooperad $Y$ as a conilpotent copresheaf $N$ on $\RR$ with structure maps relating $N$ to $Y$, we are then able to show that for an $\infty$-operad $X$ the adjunction in the previous theorem can be refined as follows. 

\medskip

\begin{theoremintro}
The previous adjunction lifts to an adjunction between conilpotent $\infty$-coalgebras over $\Bar X$ and $\infty$-algebras over $X$.\\
\end{theoremintro}

\noindent (See Theorem~\ref{ThAdjAlgCoalg} for a precise formulation, including a dual statement for $\infty$-coalgebras over an $\infty$-cooperad $Y$ and $\infty$-algebras over $\coBar Y$).

\medskip

If the $\infty$-operad comes from a classical (``strict'') differential graded operad $P$, then any $P$-algebra $A$ defines an $\infty$-algebra $M=N(P,A)$ over $X$, and one can recover the usual bar construction of $A$ from $\Bar M$ by evaluation of the copresheaf at a specific tree. In this way, our bar-cobar duality for $\infty$-algebras and $\infty$-coalgebras extends in a very precise way the duality for algebras over operads and coalgebras over cooperads as discussed in~\cite[Section 11.3]{LV}).
We emphasize, however, that our proofs are different and initially apply to presheaves and copresheaves, even before adding the (co-)algebra structure maps; cf. the two theorems above.

\medskip

For every natural number $\ell \geq 2$ the corolla tree $C_\ell$ with $\ell$ leaves, pictured as
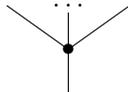
\begin{equation*}
 \xymatrix@R=12pt@C=14pt{
&\ldots&\\
&*=0{\bullet}\ar@{-}\ar@{-}[u]\ar@{-}[ul]\ar@{-}[ur]&\\
&\ar@{-}[u]&&\\
}
\end{equation*}
is an object of the category $\AA$. In~\cite{HM21} we defined the ``dendroidal homology'' $DH_*^{(\ell)}(X)$ of an $\infty$-operad $X$ as the homology of the chain complex $\Bar(X)(C_\ell)$.
For all $\ell \geq 2$ together these graded abelian groups form a (strict) cooperad $DH_*(X)$. If $X$ is a classical Koszul operad in vector spaces, this cooperad is precisely the Koszul dual of $X$, discussed (in this classical context) in~\cite{GK,GJ,LV} among others. These homology groups $DH_*^{(\ell)}(X)$ can in fact be described in very traditional terms as the homology of a pair of categories,
$$DH_*^{(\ell)}(X)=H_*(C_\ell / \AA, C_\ell /\!/ \AA ; X)$$
where $C_\ell /\!/ \AA \subseteq C_\ell / \AA $ is the full subcategory of the slice category on the non-isomorphisms $C_\ell \to T$ in $\AA$ (see~\cite[Proposition 3.7]{HM21}).
This category $C_\ell /\!/ \AA $ is closely related to the \emph{partition complex} of a set with $\ell$ elements, and hence (when $\ell$ varies) one finds a direct relation to the \emph{Lie operad} ; see~\cite{Robinson, Joyal, F, Arone-Brantner, HM-tree}.

It is natural to define in a similar way the dendroidal homology of an $\infty$-algebra $M$ over an $\infty$-operad $X$ as the homology of the bar complex of $M$ evaluated at the tree $\eta$ consisting of just one leaf and no vertices, that is
$\eta = |$.
In other words, by definition,
\begin{equation}\tag{$1$}\label{DH-intro}
DH_*(M)=H_*(\Bar(M)(\eta)).
\end{equation}
This homology has the structure of a conilpotent coalgebra over the cooperad $DH_*(X)$. Moreover, we will show in Proposition~\ref{Prop:HomologyOfCat} it can again be described in terms of the homology of categories, as 
$$DH_*(M)=H_*( \eta / \RR, \eta /\!/ \RR ; M)$$
(with the double slash again denoting the full subcategory on the non-isomorphisms).

If $M=N(P,A)$ comes from an ``ordinary'' algebra $P$ over a classical operad $P$, this definition~(\ref{DH-intro}) agrees with the homology of $A$ as defined in~\cite{LV}. It is also possible to define this homology in terms of a cotangent complex~\cite{Milles}. The idea of using an operadic bar complex to define the homology of a $P$-algebra $A$ also occurs in different contexts in works of Ching~\cite{C,C2} and Harper~\cite{Harper}.
Different notions of algebras (and coalgebras) over enriched $\infty$-operads also appear in works of Haugseng~\cite{Haugseng} and a recent preprint of Petersen, Roca i Lucio and Yalin~\cite{RocaLucioPetersenYalin}. There is so far no comparison between these notions.

\bigskip

\textbf{Acknowledgements:} We 
would like to thank our institutions, as well as ANR HighAGT (ANR-20-CE40-0016) and the Dutch science foundation NWO, for supporting several mutual visits. 


%

\section{Some categories of trees}\label{S:Categories}

\subsection{The category $\Omega$}

We recall from~\cite{HM} the category $\OOmega$ of trees. Its objects are finite trees with external edges, one of which is selected as the \emph{root}, while the others are referred to as \emph{leaves}.
The other edges, connected to two vertices, are called \emph{internal}. The \emph{valence} of a vertex is the number of incoming edges, where the edges are oriented towards the root. A vertex is called \emph{external} if it has only one internal edge attached to it. An object of $\Omega$ is called \emph{reduced} if it has no unary vertices, i.e., vertices of valence 1, and \emph{open} if it has no vertices of valence zero.

\begin{example}
 The second tree $T$ pictured below has two internal edges $c$ and $e$,
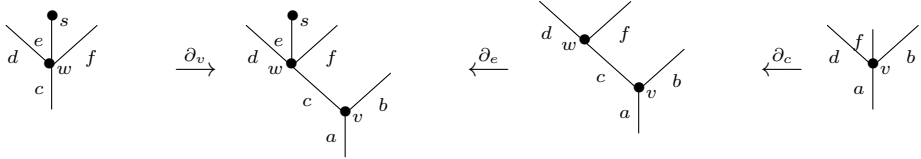
 \begin{equation*}
 \vcenter{
 \xymatrix@R=12pt@C=14pt{
&*=0{\ \, \bullet_s}&\\
&*=0{\ \, \bullet_w}\ar@{-}\ar@{-}[u]^{e}\ar@{-}[ul]^d\ar@{-}[ur]_f&\\
&\ar@{-}[u]^c&&\\
}}
\stackrel{\partial_v}{\longrightarrow} \ 
\vcenter{\xymatrix@R=12pt@C=14pt{
&\\
&*=0{\ \, \bullet_s}&&&\\
&*=0{{}_w \bullet \ \ }\ar@{-}[ul]^d\ar@{-}[u]^e\ar@{-}[ur]_f&&\\
&&*=0{\ \, \bullet_v}\ar@{-}\ar@{-}[ul]^{c}\ar@{-}[ur]_b&&\\
&&\ar@{-}[u]^a&&\\
}}
\stackrel{\partial_e}{\longleftarrow} \ 
\vcenter{\xymatrix@R=12pt@C=14pt{
&&&&\\
&*=0{{}_w \bullet \ \ }\ar@{-}[ul]^d\ar@{-}[ur]_f&&\\
&&*=0{\ \, \bullet_v}\ar@{-}\ar@{-}[ul]^{c}\ar@{-}[ur]_b&&\\
&&\ar@{-}[u]^a&&\\
}}
\stackrel{\partial_c}{\longleftarrow}
\vcenter{\xymatrix@R=12pt@C=14pt{
&&\\
&*=0{\ \, \bullet_v}\ar@{-}\ar@{-}[u]^{f}\ar@{-}[ul]^d\ar@{-}[ur]_b&\\
&\ar@{-}[u]^a&&\\
}}
\end{equation*}
 three leaves $b,d,f$ and a root $a$. The vertex $v$ of valence $2$ is external, as is the stump $s$ on top of valence zero. The tree $T$ is reduced but not open.
\end{example}

The morphisms $S \to T$ in $\Omega$ are compositions of ``elementary'' morphisms: isomorphisms, external faces $\partial_v: S\to T$ where $S$ is obtained from $T$ by chopping off an external vertex $v$ (together with the external edges attached to it), internal (or inner) faces $\partial_e:S\to T$ where $S$ is obtained from $T$ by contracting an internal edge, and degeneracies $S \to T$ where $T$ is obtained from $S$ by erasing a unary vertex (and identifying its incoming and outgoing edges). In the example above, $\partial_v$ is an external face chopping off the root vertex $v$, while $\partial_e$ and $\partial_c$ are inner faces.

A \emph{subtree} $S$ of $T$ is the source of a composition of external faces $ S\to T$. If this map preserves the root, we call the subtree $S$ and the map a \emph{pruning} of $T$. The set of leaves of a tree $T$ is usually denoted by $\lambda(T)$.
A tree with just one vertex is called a \emph{corolla}, and the unique tree with no vertices is denoted $\eta$. It has just one edge which is at the same time the leaf and the root. For more details and examples, see~\cite{HM}.

\subsection{Subcategories of $\Omega$}

In this paper, we will consider two subcategories $\AA$ and $\RR$ of $\Omega$. In fact, they are subcategories of the category $\Omega^o_r$ of open and reduced trees. Notice that the subcategory 
$\Omega^o_r$ does not have degeneracies, so all morphisms are compositions of external or internal faces and isomorphisms.

The category $\AA$ consists of open and reduced trees (so all of whose vertices have valence at least $2$.) 
The morphisms in $\AA$ are compositions of isomorphisms and inner faces. These are the morphisms $S\to T$ in $\Omega$ which preserve the root edge and the set of leaves (but they can permute the leaves).

The larger category $\RR$ also consists of open and reduced trees. The morphisms in $\RR$ are those morphisms in $\Omega$ which preserve the root edge. So these are compositions of isomorphisms, inner faces and external faces (ie. prunings) chopping off a vertex attached to leaves.

So we have embeddings
$$\AA \subseteq \RR \subseteq \Omega^o_r \subseteq \Omega.$$

\subsection{Factorisation}

Let us refer to composition of inner faces simply as inner faces also, and similarly for external faces. We will often use the fact that any morphism $\alpha: S\mono T$ in $\Omega^o_r$ factors uniquely up to isomorphism as an inner face followed by an external one.
\begin{equation*}
\xymatrix@M=8pt{
 S\ar@{>->}[rr]^{\alpha}\ar@{>->}[dr]_{\beta} & & T\\
 & R \ar@{>->}[ur]_{\gamma} & \\}
\end{equation*}  
The leaves and root of $R$ are the images under $\alpha$ of the leaves and root of $S$. The map $\beta$ is in $\AA$. If $\alpha$ is a map in $\RR$, then $\gamma$ is again in $\RR$, making $R$ a pruning of $T$.

\subsection{Inner faces and blow up}
If $\alpha: S\to R$ is a morphism in $\AA$, we can also think of $R$ as obtained from $S$ by blowing up each vertex $v$ of $S$ to a subtree of $R$ that we denote by $R_v$ or $R(\alpha)_v$. This subtree $R_v \mono R$ can be defined uniquely up to isomorphism by using the factorisation above: The vertex $v$ defines a unique external face map 
$$i_v : C_v \to S$$
from a corolla sending the unique vertex of $C_v$ to $v\in S$. The composition $\alpha \circ i_v $ now factors essentially uniquely as a map in $\AA$ followed by an external face map, as in 
\begin{equation*}
\xymatrix@M=8pt{
 C_v\ar[r]^{i_v} \ar@{-->}[d] & S \ar[d]^\alpha\\
  R(\alpha)_v \ar@{>->}[r] & R \\}
\end{equation*}  
and this defines $R(\alpha)_v$ uniquely as a subtree of $R$. Note that $R(\alpha)_v \to R$ is not a morphism in $\RR$ unless $v$ is the root vertex of $S$.

\subsection{Grafting and decomposition of trees}

We first discuss a simple case. If $e$ is an inner edge in a tree $T$, ``cutting'' the tree $T$ at $e$ results in two subtrees (external faces) $S \mono T$ and $R \mono T$ such that $e$ is a leaf of $S$ and the root of $R$. Then $T$ is the grafting of $R$ onto $S$ at $e$, denoted 
$$T= S \circ_e R.$$
(Note that of the two inclusions $S \mono T$ and $R \mono T$, only $S \mono T$ lies in $\RR$.) We also say that $S$ and $R$ form a \emph{decomposition} of $T$. This decomposition corresponds to an isomorphism class of morphisms in $\AA$ 
$$ \alpha : T_e \mono T$$
where $T_e$ is obtained by contracting all inner edges in $T$ except $e$. So $T_e$ has one inner edge only, and two vertices. In terms of the blow-ups discussed before, $S=T(\alpha)_r$ where $r$ is the root vertex of $T_e$ and $R=T(\alpha)_v$ where $v$ is the top vertex of $T_e$.

More generally, if $E$ is a family of inner edges in $T$, cutting $T$ at all these edges results in a decomposition of $T$ into smaller trees, or an iterated grafting 
\begin{equation}\tag{$\star$}
T = T_r \circ_{e \in E} T_e 
\end{equation}
where $T_r$ is the connected component containing the root edge (and root vertex), and $T_e$ is the connected component having $e$ as root edge (and the vertex of $T$ immediately above $e$ as root vertex). In terms of blow-ups, write $T_E$ for the tree obtained by contracting all inner edges not in $E$. This gives an inner face map $\alpha : T_E  \mono T$, and $E$ is exactly the set of inner edges of $T_E$. Moreover, for the vertex $v$ above $e$ in $T_E$, we have $T(\alpha)_v=T_e$, while for the root vertex $r$ of $T_E$, we have
$T(\alpha)_r=T_r$.  In this way, decompositions of a tree $T$ as in ($\star$) correspond bijectively to isomorphism classes of morphisms $S \mono T$ in $\AA$, the subtrees of $T$ constituting the decomposition being the blow-ups of vertices in $S$.

%

\section{Presheaves, (co)operads and (co)algebras}\label{DefStructures}

In this section, we review the notion of preoperad  and $\infty$-operad and their dual versions from~\cite{HM21}. In addition, we present a corresponding notion of prealgebra and algebra over a preoperad and $\infty$-operad, together with its dual coalgebraic version.

\subsection{Presheaves}

Let us fix a field $k$ and simply write $Ch$ for the category of chain complexes over $k$. A \emph{presheaf} on $\AA$ is a contravariant functor on $\AA$ with values in $Ch$, so $\AAop \to Ch$, and a \emph{copresheaf} is a covariant functor $\AA \to Ch$. We will also consider presheaves and copresheaves on the larger category $\RR$. The morphisms between such (co)presheaves are the natural transformations. A morphism is a quasi-isomorphism if each of its components is.

\begin{example}\label{2.1}
 The following examples motivate much of what follows. If $P$ is an (uncoloured) operad in $Ch$, its nerve $NP$ defines a presheaf on $\AA$, defined on an object $T$ by 
 $$NP(T)=\bigotimes_v P(in(v))$$
 where $v$ ranges over vertices of $T$ and $in(v)$ is the set of incoming edges of $v$. If $S \to T$ is a morphism in $\AA$, composition in the operad $P$ defines a map $NP(T) \to NP(S)$.
 In this context,  it is convenient to present $P$ in a ``coordinate free'' fashion, with structure maps $P(F) \otimes \bigotimes_f P(G_f) \to P(G)$ for maps $G \to F$ of finite sets, where the fibre over $f$ is denoted by $G_f$.
 The reader will notice that this nerve faithfully represents the operad $P$ when $P$ is reduced in the sense that $P(0)=0$ and $P(1)=k$.
 
 If $A$ is a $P$-algebra, then $A$ and $P$ together define a presheaf $N(P,A)$ on the larger category $\RR$, by
 $$N(P,A)(T)=NP(T) \otimes \bigotimes_\ell A=NP(T) \otimes A^{\otimes \lambda T}$$
 where $\ell$ ranges over the leaves of $T$. A morphism $S\to T$ in $\AA$ induces a map $N(P,A)(T) \to N(P,A)(S)$ as for $NP$, leaving the tensor power of $A$ alone. A morphism $\partial_w: S \to T$ chopping off a top vertex $w$ defines a map
 $$\partial_w^*: N(P,A)(T) \to N(P,A)(S)$$
 as follows.
 Write 
%
 $$N(P,A)(T) = \left( \sbigotimes_{v \neq w} P(in \, v) \otimes
  \sbigotimes_{\ell \notin in \, w} A \right) \otimes \left(  P(in \, w) \otimes \sbigotimes_{\ell \in in \, w}  A\right).$$
   The action of $P$ on $A$ gives a map $P \displaystyle{(in\,  w) \otimes\sbigotimes_{\ell \in in \, w} A \to A}$, and hence a map
%
  $$ N(P,A)(T) \to \sbigotimes_{v \neq w} P(in \, v) \otimes \sbigotimes_{\ell \notin in \, w} A \otimes A.$$
 The codomain is precisely $N(P,A)(S)$, with the last occurence of $A$ corresponding to the leaf of $S$ that was the output edge of $w$ in $T$.
\end{example}


\subsection{Preoperads and $\infty$-operads}

A preoperad is a presheaf $X$ on $\AA$ equipped with structure maps $\theta_e^{S,T}$ for each grafting $S \circ_e T$, usually simply denoted by $\theta_e$ or just $\theta$,
$$ \theta : X (S \circ_e T) \to X(S) \otimes X(T).$$
These structure maps  are required to satisfy the evident naturality conditions in $S$ and $T$, as well as associativity conditions. The latter come in two kinds, for ``nested'' graftings and ``parallel'' graftings, as pictured in
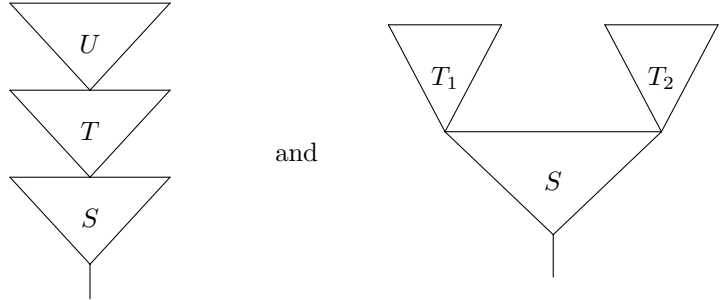
\begin{equation*}
\vcenter{\xymatrix@R=7pt@C=7pt{
&*{}\ar@{-}[rrrr]&&&&*{}\\
&&& U &&&\\
&*{}\ar@{-}[rrrr]&&*{} \ar@{-}[uull]\ar@{-}[uurr]&&*{}\\
&&& T &&&\\
&*{}\ar@{-}[rrrr]&&*{} \ar@{-}[uull]\ar@{-}[uurr]&&*{}\\
&&& S &&\\
&&&*{} \ar@{-}[uull]\ar@{-}[uurr]&& \\
&&&*{} \ar@{-}[u] \\
}}
\quad \quad \text{ and } \quad \quad
\vcenter{\xymatrix@R=10pt@C=10pt{
*{}\ar@{-}[rr]&&*{}&&*{}\ar@{-}[rr]&&*{}\\
&T_1&&&&T_2&&\\
&*{}\ar@{-}[uul]\ar@{-}[uur]\ar@{-}[rrrr]&&&&*{}\ar@{-}[uul]\ar@{-}[uur]\\
&&& S &&\\
&&&*{}\ar@{}[u] \ar@{-}[uull]\ar@{-}[uurr]&& \\
&&&*{} \ar@{-}[u] \\
}}\end{equation*}  
exactly as for the usual associativity conditions for the operadic partial compositions, see~\cite{LV}.
We will always assume that $X(\eta)$ is $k$ and $\theta_e^{S,T}$ is the canonical isomorphism whenever $S$ or $T$ is $\eta$.

For $x \in X(S \circ_e T)$, we will denote $\theta_e^{S,T}(x)$ by $x_1 \otimes x_2$ in $X(S) \otimes X(T)$, rather than  $\sum x_1 \otimes x_2$, keeping the sum implicit. This abuse of notation will also be used for all the other kinds of decomposition maps which will appear in the rest of the paper.

The morphisms between preoperads $X \to X'$ are morphisms of presheaves which respect these structure maps $\theta$. In this way, we obtain a category of preoperads, with a forgetful functor to the category of presheaves
$$ PreOp \to Psh(\AA).$$

A preoperad $X$ is called an \emph{$\infty$-operad} if each of its structure maps $X (S \circ_e T) \to X(S) \otimes X(T)$ is a quasi-isomorphism.
The category of $\infty$-operads is denoted by $\infty Op$.
Notice that if $X=NP$ is the nerve of a classical operad, these structure maps are in fact isomorphisms. In this case, we say that $X$ is a \emph{strict} operad.

\subsection{Prealgebras and $\infty$-algebras}\label{defalgebras}

Let $X$ be a preoperad. An $X$-prealgebra is a presheaf $M$ on $\RR$, equipped with a structure map $\tau^\alpha$ for each morphism $\alpha:S\to T$ in $\RR$, of the form
$$\tau^\alpha : M(T) \to X(S) \otimes M(T/ S)$$
where $M(T/ S)=\bigotimes_\ell M(T_\ell)$ with $\ell$ ranging over the leaves of $S$ and $T_\ell=T(\alpha)_\ell$ the subtree of $T$ whose root is $\alpha(\ell)$. When no confusion can arise, we will simply write $\tau$ for $\tau^\alpha$. These structure maps should again satisfy obvious naturality, unitality (that is, $\tau$ is the canonical isomorphism when $S=\eta$) and associativity conditions. In particular, if we factor $\alpha: S \to T$ as $\alpha' \circ \gamma : S \to S' \to T$ where $\gamma$ is in $\AA$ and $\alpha' : S' \to T$ is an external face (ie. a pruning), then the leaves of $S$ are those of $S'$ and $T / S = T / S'$, and
\begin{equation*}
\xymatrix@M=8pt{
 M(T)\ar[rr]^{\tau}\ar[drr]_{\tau} & & X(S')\otimes M(T/ S') \ar[d]^{\gamma^* \otimes id}\\
 & & X(S)\otimes M(T/ S)  \\}
\end{equation*}  
commutes.
Thus the structure maps $\tau$ are completely determined by those for prunings only. Also note that any pruning is a composition of faces chopping off a single vertex. In this way, we see that the presheaf $N(P,A)$ associated to an algebra over a classical operad discussed above is an example of an $N(P)$-prealgebra.

In the special case where $X$ is an $\infty$-operad we say that $M$ is an \emph{$\infty$-algebra over $X$} if for each pruning $\alpha: S \to T$ in $\RR$, the associated structure map $M(T) \to X(S) \otimes M(T / S)$ is a quasi-isomorphism.

With the evident structure preserving morphisms, these definitions give categories 
$$Prealg(X) \ \text{ and }  \ \infty Alg(X)$$
of prealgebras and $\infty$-algebras over $X$.

\subsection{Pushforward}\label{pushforward}

Let $f:X \to X'$ be a morphism of preoperads (or $\infty$-operads).
Then any prealgebra $M$ over $X$ yields one over $X'$ by composition 
of the structure maps
\begin{equation*}
\xymatrix@M=8pt{
 M(T)\ar[rr]^{\tau}\ar[drr]_{\tau} & & X(S)\otimes M(T/ S) \ar[d]^{f \otimes id}\\
 & & X'(S)\otimes M(T/ S) \\}
\end{equation*}  
We will also write $f_!(M)$ or $f_!(M,\tau)$ for $M$ as $X'$-prealgebra with structure maps $\tau'=(f\otimes id) \circ \tau$. Notice that if $M$ is an $\infty$-algebra, $f_!(M)$ is one only if $f$ is a quasi-isomorphism.
This functor $f_!$ has a right adjoint $f^*$ describing the change of base for prealgebras, as $X-$prealgebras can be seen as coalgebras over a comonad induced by $X$, see details in Appendix~\ref{comonads}. 
Moreover in case $f$ is a quasi-isomorphism, it is shown there that $f_!$ and $f^*$ are inverse to each other up to quasi-isomorphism.

The definitions given above have evident duals, which we briefly summarize as follows.

\subsection{Precooperads and $\infty$-cooperads}\label{precoop}
A precooperad is a covariant functor (copresheaf) $Y: \AA \to Ch$, again equipped with structure maps 
$$\theta : Y(S\circ_e T) \to Y(S) \otimes Y(T)$$
satisfying naturality and associativity conditions. 
We will always assume that $Y(\eta)$ is $k$ and $\theta$ is the canonical isomorphism whenever $S$ or $T$ is $\eta$.
Such a precooperad is called an $\infty$-cooperad if all structure maps are quasi-isomorphisms. 
With the evident structure preserving morphisms, this defines categories 
$$Precoop\ \text{ and }  \ \infty Coop$$
of precooperads and $\infty$-cooperads.

\subsection{Coalgebras}
A precoalgebra over a precooperad $Y$ is a copresheaf $N: \RR \to Ch$ equipped with structure maps, one for each pruning map $\alpha : S\to T$ in $\RR$,
$$\nabla=\nabla^\alpha : N(T) \to Y(S) \otimes N(T/ S)$$
These morphisms are again required to satisfy naturality, unitality and associativity conditions. Unlike the case of algebras, these do not determine $\nabla^\alpha$ for any $\alpha:S \to T$ in $\RR$.
A precoalgebra $N$ over an $\infty$-cooperad $Y$ is called an $\infty$-coalgebra if for each pruning the structure map  $\nabla^\alpha$ is a quasi-isomorphism.
With the evident structure preserving morphisms, these give categories 
$$Precoalg(Y) \ \text{ and }  \  \infty Coalg(Y)$$
of precoalgebras and $\infty$-coalgebras over $Y$.
As for prealgebras, any morphism $f: Y \to Y'$ of precooperads again gives an evident pushforward functor 
$$f_! : Precoalg(Y) \to Precoalg(Y').$$

\subsection{Examples}\label{2.7}
As for operads and their algebras, a ``classical'' cooperad $C$ yields an $\infty$-cooperad $NC$. If $C$ is given by coordinate-free decomposition maps 
$$C(G) \to C(F) \otimes \sbigotimes_{f\in F} C(G_f)$$
for maps $G\to F$ of finite sets, then $NC$ can be defined as 
$$NC(T)=\sbigotimes_v C(in \ v),$$
dual to the case for operads.
The decomposition maps of $C$ make $NC$ into a covariant functor on $\AA$, and the structure maps $\theta$ are isomorphisms. Similarly, if $E$ is a coalgebra over $C$ with structure maps $E\to C(F) \otimes E^{\otimes F}$ for finite sets $F$, then we  can define an $\infty$-coalgebra $N(C,E)$ over $NC$ by
$$N(C,E)(T)=NC(T) \otimes \sbigotimes_\ell E$$
where $\ell$ ranges over the leaves of $T$.

Notice that every $\infty$-cooperad $Y$ is quasi-isomorphic to a strict cooperad. Indeed, let $C_Y$ be the cooperad defined by $C_Y(n)=Y(C_n)$ where $C_n$ is the corolla with $n$ leaves.
Then the structure maps $\theta:Y(S \circ T) \to Y(S) \otimes Y(T)$ and the copresheaf structure maps for $C_{k+\ell-1} \to C_k \circ_i C_\ell$ (where $i=1,\ldots,k$) together induce  
$$C_Y(k+\ell-1)=Y(C_{k+\ell-1}) \to Y(C_k \circ C_\ell) \to Y (C_k) \otimes Y(C_\ell) = C_Y(k) \otimes C_Y(\ell)$$ 
and make $C_Y$ into a strict cooperad. The $\infty$-cooperad structure of $Y$ yields a map $Y \isomap N C_Y$, showing that $Y$ is equivalent to the strict cooperad $C_Y$.

Similarly, if $E$ is an $\infty$-coalgebra over $Y$, the complex $E(\eta)$ has the structure of a (strict) coalgebra over the cooperad $C_Y$ by the maps
$$E(\eta) \to E(C_n) \stackrel{\nabla}{\to} Y(C_n) \otimes E(\eta)^{\otimes \lambda C_n} = C_Y(n)  \otimes E^{\otimes n}.$$
Here the first map is given by the copresheaf structure of $E$ for the root inclusion $\eta \to C_n$, and the second one is given by the $\infty$-coalgebra structure of $E$ for the identity maps $C_n\to C_n$.


\section{Review of Koszul duality for $\infty$-operads}\label{Section:Operads}

In this section, we will review the bar and cobar constructions from~\cite{HM21}. These constructions involve a functor 
$$\Bar: Psh(\AA) \to CoPsh(\AA)$$
and its left adjoint
$$\coBar: CoPsh(\AA) \to Psh(\AA).$$
The properties of these functors proved in~\cite{HM21} are the following.
\begin{theorem}\label{OperadicTheorem}
 \begin{enumerate}
 \item[(i)] The functors $\Bar$ and $\coBar$ are quasi-inverse to each other. In other words, each counit $\coBar \Bar X \to X$ is a quasi-isomorphism, as is each unit $Y \to \Bar \coBar Y$.
 \item[(ii)] The functor $\Bar$ sends preoperads (resp. $\infty$-operads) to precooperads (resp. $\infty$-cooperads), and vice versa for $\coBar$.
 \item[(iii)] The adjunction lift to adjunctions
 $$ \coBar :  PreCoop  \rightleftarrows PreOp : \Bar$$
 $$ \coBar :  \infty Coop \rightleftarrows  \infty Op : \Bar$$
  modulo suspension.
\end{enumerate}
\end{theorem}
We will now review the definitions involved.
Note that the sign and degree conventions slightly differ from the article~\cite{HM21}. They are now more in line with the Koszul sign rules.
 
 \subsection{The bar complex of a presheaf}\label{BarOperadInv}
 Let $X$ be a presheaf on $\AA$ (always with values in $Ch$).
 The \emph{bar complex} $\Bar X$ of $X$ is the copresheaf $\Bar X : \AA \to Ch$ defined on a tree $S$ by
 $$\Bar (X)(S)= \Big( \prod_{\alpha:S\to T} \Hom(k[en_T],X(T))\Big)^{inv},$$
 where $\alpha$ ranges over the finite set of morphisms in $\AA$ with domain $S$, and $en_T$ is the set of enumerations of inner edges of the tree $T$. We will denote such an enumeration
 $$e=(e_1, \ldots, e_p)$$
 where each $e_i$ is taken to have degree $-1$.
So $e$ has degree $-p$, and we usually write just $e$ in contexts like $(-1)^e$. The invariants are taken for permutations of these enumerations and isomorphism $T \isomap T'$ under $S$. Thus, an element $\omega \in \Bar(X)(S)$ homogeneous of degree $n$ can be written as 
$$\omega=(\omega_\alpha)_{\alpha: S\to T}$$
where $\omega_\alpha$ assigns to each enumeration $e=(e_1, \ldots, e_p)$ an element $\omega_\alpha(e)$ in $X(T)$ of degree $n-p$. Invariance means that, first,
$$\omega_\alpha(e \cdot \sigma)=(-1)^\sigma \omega_\alpha(e)$$
for any permutation $\sigma \in \Sigma_p$, where $(-1)^\sigma$ is the sign of $\sigma$ ; and secondly, for $S \stackrel{\alpha}{\to} T \stackrel{\gamma}{\to} T'$ where $\gamma$ is an isomorphism,
$$\gamma^*\omega_{\gamma\alpha}(\gamma e)=\omega_\alpha(e).$$
Here $\gamma e$ is the induced enumeration of the inner edges of $T'$, and $\gamma^* :X(T')\to X(T)$ is given by the presheaf structure on $X$.

The differential on $\Bar(X)(S)$ is defined for $\alpha : S \to T$ by
$$(\partial \omega)_\alpha (e) =
\partial_X \omega_\alpha (e) + (-1)^\omega \int_d d^* \omega_{d \circ \alpha} (de).$$
Here $\partial_X$ is the differential of $X(T)$, and the integral is the ``weighted'' sum over all morphisms $d:T\to T'$ in $\AA$ where $T'$ has exactly one more inner edge than $T$. This morphism as well as the new inner edge are denoted by $d$. The ``weighted'' sum involves averaging over isomorphism classes of such morphisms $T \stackrel{d}{\to} T'$ (for $T$ fixed), and is a way of avoiding the unnatural choice of representatives for isomorphism classes, cf. Appendix of~\cite{HM21}.
Furthermore, we have written $(-1)^\omega$ as short for $(-1)^{degree(\omega)}$, as elsewhere in this paper.

The copresheaf structure on $\Bar X$ is simply given by composition in $\AA$. So for $\beta:R\to S$ and $\omega \in \Bar(X)(R)$, one defines $\beta_*(\omega)$  on $\alpha:S\to T$ by
$$\beta_*(\omega)_\alpha=\omega_{\alpha \beta}.$$
It is straightforward to check that $\partial \partial=0$ and that $\beta_*$ preserves the differential.

It will sometimes be more convenient to use a different notation for $\Bar X$. Notice that
$$\prod_\alpha \Hom(k[en_T],X(T)) \cong \prod_{\alpha,e} X(T)[e]$$
where the product on the right is over pairs $\alpha:S \to T$ and enumerations $e$ of inner edges in $T$. The shift $X(T)[e]$ is
$$X(T)[e]_n=X(T)_{n-p}$$
where $e$ also stands for $p$ if $e=(e_1, \ldots, e_p)$. So we can also write 
$$\Bar(X)(S) = \Big(\prod_{\alpha, e} X(T)[e]\Big)^{inv}.$$
%
%

Notice that the construction $\Bar(X)$ is obviously functorial in $X$. Moreover if $X \to X'$ is a quasi-isomorphism, then so is $\Bar(X) \to \Bar(X')$.

\subsection{The bar complex by coinvariants}\label{BarOperadCoinv}
We shall need a description of $\Bar X$ in terms of coinvariants, based on the canonical isomorphism
$$\rho : \bigoplus_{\alpha,e} X(T)[e] \to \prod_{\alpha,e} X(T)[e]$$
which induces an isomorphism from coinvariants to invariants
\begin{equation}\tag{$\star \star$}\label{Eq.rho1}
\rho : \Big(\bigoplus_{\alpha,e} X(T)[e]\Big)_{coinv} \to \Big(\prod_{\alpha,e} X(T)[e]\Big)^{inv}.
\end{equation}
We shall write elements on the left as triples $(\alpha,e,x)$, or sometimes simply as $(e,x)$ when $\alpha$ is understood. Such an element $(e,x)$ has degree $e+x$ (i.e. the sum of the length of $e$ and of the degree of $x$).
The isomorphism $\rho$ is completely determined by invariance together with $\rho(\alpha,e,x)_\alpha(e)=x$ and $\rho(\alpha,e,x)_\beta(f)=0$ if $\beta$ is not isomorphic to $\alpha$.
We shall write $\rho$ as an isomorphism 
$$\rho: \bar\Bar(X) \to \Bar(X)$$
where
$$\bar\Bar(X)(S)=\Big(\bigoplus_{\alpha,e} X(T)[e]\Big)_{coinv}$$
as in (\ref{Eq.rho1}), and refer to $\bar\Bar(X)$ as ``the bar complex of $X$ in terms of coinvariants''.
Transporting the differential and the presheaf structure along $\rho$ gives the following formulas for $\bBar(X)$:
For $\beta:R\to S$ the map $\beta_*:\Bar(X)(R)\to \Bar(X)(S)$ is given by
$$\beta_*(\gamma,e,x)= \left\{
\begin{array}{ll}
 (\alpha,e,x) & \text{ if } \gamma=\alpha\beta\\
 0 & \text{ otherwise.}
\end{array}
\right.
$$
(If $\gamma:R \to T$ factors through $\beta$ as $\gamma=\alpha \beta$ then $\alpha$ is unique.) The differential $\bar\partial$ on $\bBar(X)(S)$ is given by
$$\bar\partial(\alpha,e,x)=(\alpha,e,\partial_X (x))
+(-1)^{e+x} \sum_{i \notin im(\alpha)} (-1)^{i-1} (e_i^{-1}\alpha, \partial_ie,e_i^*x).$$
Here $\partial_ie=(e_1, \ldots, \hat{e_i}, \ldots, e_p)$ if $e=(e_1, \ldots, e_p)$, and we have used the notation
\begin{equation*}
\xymatrix@M=8pt{
 S\ar[dr]_{\alpha}\ar[r]^{e_i^{-1}\alpha} & \partial_{e_i}T \ar[d]^{e_i}\\
 &  T  \\}
\end{equation*}  
for the factorisation of $\alpha$ through the tree $\partial_{e_i}T$ obtained from $T$ by contracting the edge $e_i$. Such a factorisation exists and is unique if the edge $e_i$ does not belong to the image of $\alpha$.

\bigskip

Next, we will briefly review the dual construction.

\subsection{The cobar complex of a copresheaf}
Let $Y: \AA \to Ch$ be a copresheaf on $\AA$. We define for each tree $S$ a complex
$$\coBar(Y)(S)=\Big(\bigoplus_{\alpha: S\to T} Y(T) \otimes k[en_T]\Big)_{coinv}$$
with coinvariants as before. The differential on $\coBar(Y)$ is defined by
$$\partial (\alpha,y \otimes e) = (\alpha,\partial_Y y \otimes e) +
(-1)^y \int_d (d \circ \alpha, (\partial_d)_*y \otimes de)$$
where $d$ runs over morphisms $T \stackrel{d}{\to}T'$ in $\AA$ adding exactly one inner edge, again called $d$ (as before). The degree of $y \otimes e$ in the summand for $\alpha$ is $y-e$. The presheaf structure is simply given by composition: For $R \stackrel{\beta}{\to} S$ and $S \stackrel{\alpha}{\to}T$, 
$$\beta^*(\alpha, y \otimes e)=(\alpha \beta, y \otimes e).$$
This clearly respects the differential. Just like for the bar complex $\Bar$, the cobar complex $\coBar(Y)$ is functorial in $Y$ and the functor $\coBar$ preserves the quasi-isomorphisms.

\subsection{Adjointness}\label{defphipsi}
The functor $\coBar$ is left adjoint to $\Bar$.
Explicitly, for a presheaf $X$ and a copresheaf $Y$, there is a bijective correspondence between morphisms
$$ \coBar Y \stackrel{\varphi}{\to}X \quad \text{ and } \quad Y \stackrel{\psi}{\to} \Bar X $$
of which we will need an explicit description later.
Given $\varphi$, one defines $\psi$ for $y \in Y(S)$ and $S\stackrel{\alpha}{\to}T$ by
$$\psi_S(y)_\alpha(e)=(-1)^e\varphi_T(1_T, \alpha_*(y) \otimes e). $$
And conversely, given $\psi$, one defines $\varphi$ for $S\stackrel{\alpha}{\to}T$ and $y\otimes e$ by
$$\varphi_S(\alpha,y\otimes e)=(-1)^e\alpha^*\psi_T(y)_{1_T}(e).$$
Given $\varphi$, the corresponding $\psi$ can also be written in terms of coinvariants as
$$\bar \psi : Y \to \bBar X,$$
$$\bar \psi_S(y)=\int_{\alpha,e}(-1)^e(\alpha,e, \varphi_T (1_T,\alpha_*(y),e)).$$
One easily checks that $\varphi$ preserves the differential if and only if $\psi$ does, that $\varphi$ is natural if and only if $\psi$ is, and that the passages from $\varphi$ to $\psi$ and back are mutually inverse (see Appendix~\ref{AppendixSigns} for some details).

\subsection{The bar complex of a preoperad}\label{barcomplexpreoperad}
Suppose $X$ has the structure of a preoperad, given by maps
$$\theta_a=\theta_a^{S_1,S_2} : X (S_1 \circ_a S_2) \to X(S_1) \otimes X(S_2).$$
Then $\Bar(sX)$ carries the structure of a precooperad, where $sX$ is the suspension of $X$ with elements denoted $sx \in (sX)_n$ for $x \in X_{n-1}$.
We wish to give an explicit description of this structure. First, we will describe how $\theta$ induces a natural map $\Delta^a$ as in
\begin{equation*}
\xymatrix@M=8pt{
 \displaystyle{\prod_{\alpha,e}}X(T)[e]\ar[drrr]_{\Delta'}\ar@{-->}[rrr]^{\Delta^a} &&&
 \displaystyle{\prod_{\alpha_1,e_1}}X(T_1)[e_1] \otimes \displaystyle{\prod_{\alpha_2,e_2}}X(T_2)[e_2] \ar[d]_{\mu}^\wr\\
 &&&  \displaystyle{\prod_{\alpha_1, \alpha_2,e_1,e_2}}X(T_1)\otimes X(T_2)[e_1+e_2]  \\}
\end{equation*}  
Here the grafting $S_1 \circ_a S_2$ induces a bijection between extensions $\alpha : S \to T$ and pairs of such $\alpha_1 : S_1 \to T_1$ and $\alpha_2 : S_2 \to T_2$ where $T=T_1 \circ_a T_2$.
If $ae$ is an enumeration of the inner edges of $T$ where we put the grafting edge $a$ up front, the corresponding enumerations of $T_1$ and $T_2$ are denoted $e_1$ and $e_2$. The isomorphism $\mu$ is given by the usual Koszul sign
$$\mu(\omega_1 \otimes \omega_2)(e_1,e_2)=(-1)^{e_1 \omega_2}\omega_1(e_1) \otimes \omega_2(e_2)$$
where we have suppressed the $\alpha$'s and have written $e_1\omega_2$ for the product of the corresponding degrees, as usual.
The map $\Delta'$ is defined by
$$\Delta'(\omega)(e_1,e_2)=(-1)^\omega \theta_a(\omega(ae_1e_2))$$
where $\theta_a$ is the preoperad structure map of $X$ for the grafting $T=T_1 \circ_a T_2$.
Notice that $\Delta'$ is of degree $-1$ (due to dropping the edge $a$).
This map $\Delta^a$ is natural and well-defined on invariants, so defines a map
$$\Delta^a : \Bar(X)(S) \to \Bar(X)(S_1) \otimes \Bar(X)(S_2)$$
of degree -1.
To check that $\Delta^a$ is coassociative and respects the differential, it is convenient to give a description of the corresponding map on coinvariants, denoted 
$$\bar\Delta^a : \bBar(X)(S) \to \bBar(X)(S_1) \otimes \bBar(X)(S_2).$$
To this end, consider the diagram
\begin{equation*}
\hspace{-2cm}
\xymatrix@M=10pt@R=26pt{
 \displaystyle{\prod_{\alpha,e}}X(T)[e]\ar[r]^{\Delta^a\quad \quad \quad } &
 \displaystyle{\prod_{\alpha_1,e_1}}X(T_1)[e_1] \otimes \displaystyle{\prod_{\alpha_2,e_2}}X(T_2)[e_2] \ar[r]^{\mu}_\sim
 &
 \displaystyle{\prod_{\alpha_1, \alpha_2,e_1,e_2}}X(T_1)\otimes X(T_2)[e_1+e_2] 
 \\
 \displaystyle{\bigoplus_{\alpha,e}}X(T)[e]\ar@{-->}[r]^{\bar\Delta^a \quad \quad \quad }\ar[u]_{\rho} &
 \displaystyle{\bigoplus_{\alpha_1,e_1}}X(T_1)[e_1] \otimes \displaystyle{\bigoplus_{\alpha_2,e_2}}X(T_2)[e_2] \ar[r]^{\bar \mu}_\sim \ar[u]_{\rho \otimes \rho}
 &
 \displaystyle{\bigoplus_{\alpha_1, \alpha_2,e_1,e_2}}X(T_1)\otimes X(T_2)[e_1+e_2] \ar[u]_{\rho}
 \\}
\end{equation*}  
where $\rho$ is the standard isomorphism, and the isomorphism $\bar \mu$ is the one making the right hand square commute, so
$$\bar\mu ((e_1,x_1) \otimes (e_2,x_2)) =
(-1)^{e_1(e_2+x_2)} (e_1, e_2,x_1 \otimes x_2)$$
(We have not specified (co)invariants in the above diagrams in order to save space.)
Then one finds the following formula for $\bar \Delta^a$ making the left hand square commute,
$$\bar\Delta^a(e,x)=
(-1)^{e+x+e_1(e_2+x_2)} ((e_1,x_1) \otimes (e_2,x_2)),$$
with $e+x$ in the sign coming from $\Delta'$ and $e_1(e_2+x_2)$ from $\bar\mu$, and where we have used the notation $\theta_a(x)=x_1 \otimes x_2$ (recall that we keep the sum implicit). Here we assume $e$ is given in the order $ae_1e_2$.
With these formulas, it is now straightforward to check that $\bar \Delta^a$ is coassociative, and that $\bar \partial$ is a coderivation with respect to $\bar \Delta^a$, as in
$$\bar \Delta^a \bar \partial=-(\bar \partial \otimes 1 + 1 \otimes \bar \partial) \bar \Delta^a,$$
where the minus sign comes from $\bar \Delta^a$ being of degree $-1$.
(See the appendix for a verification.)
After suspending,  we obtain a coassociative coproduct of degree $0$,
$$\bar\Delta^a : \bBar(sX)(S) \to \bBar(sX)(S_1) \otimes \bBar(sX)(S_2)$$
for which $\bar\partial$ is a coderivation, showing that $\bBar(sX)$ has the structure of a precooperad. It is clear that $\bBar(sX)$ is in fact an $\infty$-cooperad if $X$ is an $\infty$-operad and a strict cooperad if $X$ is a strict operad.

\subsection{The cobar complex of a precooperad}

Suppose the copresheaf $$Y: \AA \to \Ch$$ has the additional structure of a precooperad given by maps
$$\theta_a : Y(S) \to Y(S_1) \otimes Y(S_2)$$
for each grafting $S=S_1 \circ_1 S_2$, as in Section~\ref{precoop}. This induces a coproduct
$$\nabla=\nabla_a: \coBar(Y)(S) \to \coBar(Y)(S_1) \otimes \coBar(Y)(S_2),$$
this time of degree $+1$. 
For an element $(\alpha, y, e)$ in $\coBar(Y)(S)$, where $\alpha: S\to T, y \in Y(T)$ and $e$ an enumeration of $T$, taken in the order $e=ae_1e_2$ as before,
$$\nabla(\alpha,y,e)=(-1)^{y+e_1y_2} (\alpha_1, y_1, e_1) \otimes (\alpha_2, y_2, e_2).$$
The sign here is the usual Koszul sign, where we view $\nabla$ as the composition
$$\nabla=tw \circ (\theta_a  \otimes id) \circ (id \otimes drop_a)$$
where $drop_a$ maps $e=ae_1e_2$ to $e_1e_2$, while $\theta_a$ maps $y$ to $y_1 \otimes y_2$ and $tw$ shuffles $(y_1\otimes y_2, e_1e_2)$ to $(y_1, e_1) \otimes (y_2, e_2)$ (where we again suppress $\alpha, \alpha_1, \alpha_2$).
As the degree of $(\alpha,y,e)$ in $\coBar(Y)$ is $y-e$, the map $\nabla$ is of degree $+1$. It can again be checked that $\nabla$ is natural in $S$, (graded) coassociative, and that the differential
$$\partial(\alpha, y, e)= (\alpha, \partial_Y y, e)+(-1)^y \int_d
(d \circ \alpha, d_*y, de)$$
is a graded coderivation, i.e.
$$\nabla \circ \partial = - (\partial \otimes 1 + 1 \otimes \partial) \circ \nabla.$$
Thus, if $Y$ is a precooperad, the cobar construction $\coBar(s^{-1}Y)$ of the desuspension $s^{-1}Y$ will have a coproduct $\nabla : \coBar(s^{-1}Y)(S) \to \coBar(s^{-1}Y)(S_1) \otimes \coBar(s^{-1}Y)(S_2)$ of degree $0$ and makes $\coBar(s^{-1}Y)$ into a preoperad. It is evident from the definition of $\nabla$ that $\coBar(s^{-1}Y)$ is in fact an $\infty$-operad if $Y$ is an $\infty$-cooperad, and a strict operad if $Y$ is a strict cooperad.

\subsection{Adjointness and structures}
The bijective correspondence between maps of presheaves ${\coBar Y \stackrel{\varphi}{\to}X}$ and maps of copresheaves
$ Y \stackrel{\psi}{\to}\Bar X$ of Section~\ref{defphipsi} now restricts to preoperads and precooperads (or $\infty$-operads and $\infty$-cooperads, respectively).
More precisely, if $(X, \theta_X)$ is a preoperad and $(Y, \theta_Y)$ is a precooperad, and $\varphi : \coBar Y \to X$ and $\psi: Y \to \Bar X$ are maps of degree $-1$ that correspond under the adjunction, then 
$$(\varphi \otimes \varphi) \nabla=-\theta_X \varphi \quad \quad \text{ if and only if } \quad \quad (\psi \otimes \psi) \theta_Y = - \Delta \psi$$
(see Appendix~\ref{AppendixSigns}).
So if the maps $\coBar(s^{-1}Y) \stackrel{\varphi'}{\to}X$ and $Y \stackrel{\psi'}{\to} \Bar(sX)$ of degree $0$ correspond under the adjunction, then applying the equivalence to $\varphi = \varphi' \circ (\coBar s^{-1})$ and $\psi = (\Bar s)^{-1} \circ \psi'$ shows that $\varphi'$ is a map of preoperads if and only if $\psi'$ is one of precooperads. This gives the adjunction as stated in the beginning of this section.

\subsection{Remark (Strictification)}\label{Strictification}

Recall from Section~\ref{2.7} that every $\infty$-cooperad $Y$ is quasi-isomorphic to a strict one that we denoted $C_Y$.
Since the functor $\coBar$ preserves quasi-isormorphisms, we deduce from Theorem~\ref{OperadicTheorem} that every $\infty$-operad is quasi-isomorphic to a strict one by considering the zig-zag
$$X\stackrel{\sim}{\longleftarrow} \coBar(s^{-1} \Bar (sX)) \stackrel{\sim}{\longrightarrow} \coBar(s^{-1}C_{\Bar(sX)}).$$

%

\section{The bar-cobar adjunction for presheaves on $\RR$}\label{AlgBarETCobar}

Recall from Section~\ref{defalgebras} that algebras over a preoperad $X$ are defined as presheaves on $\RR$ with structure maps relating these to $X$.
Dually, coalgebras over a precooperad $Y$ are copresheaves on $\RR$ with such structure maps. In this section, we will again define bar and cobar functors $\Bar$ and $\coBar$ and prove that these are adjoint.
In a next section, we will show that they are in fact mutually inverse up to quasi-isomorphism.

\subsection{The bar complex of a presheaf}
Let $M$ be a presheaf on $\RR$. For an object $S$ of $\RR$, we define
$$\Bar(M)(S)=\Big(\prodc_{\alpha:S\to T} \Hom(k[en^+_T], M(T))\Big)^{inv}.$$
Here $\alpha:S \to T$ ranges over morphisms in $\RR$. An important difference with the bar construction for presheaves on $\AA$ (as in Section~\ref{BarOperadInv}) is that for a fixed $S$ it is no longer a finite set. The superscript $c$ on the product indicates that we take the subspace of the product consisting of functions with finite support only. The set of enumerations of the root and inner edges  of $T$ is denoted by $en^+_T$ (the root now has to be considered in addition to the inner edges, as the external root face of a corolla is fundamental in the algebra structure). By definition, $en^+(\eta)$ is the empty set (that is the unique edge of $\eta$ has to be thought of as a leaf in this context).
So an element $\omega$ of $\Bar(M)(S)$ assigns to each $\alpha:S \to T$ and each enumeration $e=(e_1, \ldots, e_p)$ an element $\omega_\alpha(e)$ or just $\omega(e)$ in $M(T)$.
Moreover, $\omega_\alpha(e)=0$ for all but finitely many $\alpha$'s. 
Such an element $\omega$ is required to be invariant under permutations of the enumeration $e$ and under isomorphisms under $S$, as before. Explicitly, 
$$\omega_\alpha(e\sigma) = (-1)^\sigma \omega_\alpha(e) $$
for any permutation $\sigma \in \Sigma_{p}$, and
$$\omega_\alpha(e)  = \gamma^* \omega_{\gamma  \alpha}(\gamma e)$$
for any isomorphism $\gamma : T \to T'$.
An enumeration $e=(e_1, \ldots, e_p)$ is taken to be of degree $-p$, sometimes abusively denoted by $-e$. So if $\omega$ is of degree $n$, it sends $\alpha$ and $e$ to $\omega_\alpha( e) \in M(T)_{n-p}$. 
We will also write
$$\Bar(M)(S)= \Big( \prodc_{\alpha,e}M(T)[e] \Big)^{inv}.$$
To describe the differential, recall first that an inner edge $a$ of $T$ defines a map $\partial_a : \partial_a T \mono T$ in $\AA$ (and hence in $\RR$), sometimes also simply denoted $a: \partial_a T \mono T$. And a top vertex $v$ in $T$ defines a map $\partial_v : \partial_v T \to T$ in $\RR$ where $\partial_v T$ is the tree obtained from $T$ by chopping off the vertex $v$. Now for $\omega \in \Bar(M)(S)$ the differential is given for $\alpha: S \to T$ and an enumeration $e$ by
$$(\partial \omega)_\alpha (e) = \partial \omega_\alpha(e) +(-1)^\omega \int_{d} \partial_d^* \omega_{d \alpha} (de) - (-1)^\omega \int_{v \ top} \partial_v^* \omega_{d \alpha} (de).$$
This formula resembles the one for preoperads, but there are some important differences. 
In the first integral, $d$ ranges over inner morphisms $d:T\to T'$ in $\RR$ that add one more vertex with output edge $d$, so that $T$ is obtained from $T'$ by contracting an inner edge $d$ in $T'$. In the second integral, $v$ ranges over top vertices of $T'$ with output edge $d$ which is a leaf of $T$. So
$T$ is obtained from $T'$ by chopping off this vertex $v$ in $T'$.
The external morphisms  $\partial_v: T \to T'$ of the second kind form an infinite set, so it is essential that we have finite support for the second integral to make sense. This second integral was not present in the preoperadic case, as only inner maps were considered.
One can check that $\partial \partial=0$, making $\Bar(M)(S)$ into a chain complex (it will actually be easier to check this after passing to coinvariants, as described in Section~\ref{BarAlgCoinv}).
The copresheaf structure on $\Bar(M)$ is given by composition:
For $\beta: S \to R$ and $\omega \in \Bar(M)(S)$, we set
$$ \beta_*(\omega)_\alpha(e)=\omega_{\alpha \beta}(e),$$
for $\alpha : R \to T$ and $e \in en^+_T$.
This map $\beta_*: \Bar(M)(S) \to \Bar(M)(R)$ preserves the differential and gives $\Bar(M) : \RR \to \Ch$ the structure of a copresheaf.

\subsection{The bar construction by coinvariants}\label{BarAlgCoinv}
We will need a description of $\Bar(M)$ in terms of coinvariants, as for $\Bar(X)$ in Section~\ref{BarOperadCoinv}.
Write
$$ \bBar (M) (S)= \Big(\bigoplus_{\alpha:S \to T} k[en^+_T]\otimes M(T)\Big)_{coinv}=\Big(\bigoplus_{\alpha,e} M(T)[e]\Big)_{coinv}$$
and $\rho : \bBar(M)(S) \isomap \Bar(M)(S)$ for the canonical isomorphism (remember elements of $\Bar(M)$ have finite support).
In order to describe the way the differential of $\Bar(M)$ transports along $\rho$, it will be convenient to introduce some notation.
Suppose $e=(e_1, \ldots, e_p)$ is an enumeration of the root and the inner edges of a tree $T$.
An element in $\bBar (M) (S)$ is represented by a triple $(\alpha,e,x)$, with $x \in M(T)_n$, of degree $n+e=n+p$.
Let $v_i$ be the vertex immediately above $e_i$.
The presheaf structure of $M$ gives maps 
$$e_i^* : M(T) \to M (\partial_{e_i}T)$$
for $e_i$ an inner edge and if $v_i$ happens to be a top vertex, also
$$v_i^* : M(T) \to M (\partial_{v_i}T).$$
Now define for $\alpha: S \to T, e \in en^+_T$ and $x\in M(T)$ elements $e_i^*(\alpha,e,x)$ and $v_i^*(\alpha,e,x)$ by
\begin{eqnarray*}
 e_i^*(\alpha,e,x)&=&0  \text{ if $e_i$ is an inner edge of $S$ or the root}\\
v_i^*(\alpha,e,x)&=& 0 \text{ unless $v_i$ is a top vertex of $T$ that does not belong to $S$}\\
e_i^*(\alpha,e,x)& = & (e_i^{-1}\alpha, (e_1, \ldots, \hat{e_i}, \ldots, e_p), e_i^*x) \\
& & \, \text{ if $e_i$ is an inner edge of $T$, not belonging to $S$}\\
v_i^*(\alpha,e,x)&=&(v_i^{-1}\alpha, (e_1, \ldots, \hat{e_i}, \ldots, e_p), v_i^*x) \\
& & \, \text{ if $v_i$ is a top vertex of $T$, not belonging to $S$.}
\end{eqnarray*}
Here we have used the notation
\begin{equation*}
\xymatrix@M=10pt{
 S\ar[r]^{e_i^{-1}\alpha} \ar[dr]_{\alpha}& \partial_{e_i} T \ar[d]^{e_i}\\
 & T \\}
\end{equation*}  
for $e_i$ not in $S$, and similarly for $v_i$. 
Now define the differential $\bar \partial$ on $\bBar(M)(S)$ by the formula
$$\bar\partial(\alpha,e,x)=(\alpha, e, \partial_Xx)
+(-1)^{e+x}\sum_{i=1}^p(-1)^{i-1}(e_i^*(\alpha,e,x)-v_i^*(\alpha,e,x))$$
where $\alpha:S\to T$, $e=(e_1, \ldots, e_p)\in en^+_T$ and $x \in M(T)$, as before.
Notice that in the expression after the summation for a fixed $i$ there can be no, one or two non-zero terms: one term if $e_i$ is inner in $T$ and $v_i$ non top in $T$ if in addition $e_i$ does not belong to $S$; one term if $e_i$ is a leaf of $S$ and $v_i$ is a top vertex in $T$; and two terms if $v_i$ is a top vertex in $T$ and $e_i$ does not belong to $S$. 
We will also use the notation
$$\hat \partial_{i}(\alpha,e,x)=e_i^*(\alpha,e,x)-v_i^*(\alpha,e,x)$$

The fact that $\bar\partial$ is a differential on $\bBar(M)(S)$, i.e. that $\bar\partial \bar\partial=0$, can be verified by checking that $\hat \partial_{i} \hat \partial_{j}=\hat \partial_{j}\hat \partial_{i}$.
This latter equation is non-trivial if $e_i$ is an input edge of $v_j$
 as in the picture below,  
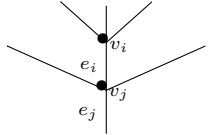
\[ 
 \xymatrix@R=12pt@C=14pt{
&&&&& \\
 &&*=0{\ \, \bullet_{v_i}}\ar@{-}[u]\ar@{-}[ul]\ar@{-}[ur]&&\\
&&*=0{\ \, \bullet_{v_j}}\ar@{-}[u]^{e_i}\ar@{-}[ull]\ar@{-}[urr]&\\
&&\ar@{-}[u]^{e_j}&&\\
}
\]
and is then based on the dendroidal identity $\ \partial_{v_j} \partial_{e_i}=\partial_{v_j} \partial_{v_i}$ as in the commutative diagram below.

\begin{equation*}
\xymatrix@M=10pt{
 \partial_{v_j} \partial_{v_i}T\ar[r]^{\partial_{v_j}} \ar@{=}[d] & \partial_{v_i}T \ar[r]^{\partial_{v_i}} & T\ar@{=}[d]\\
\partial_{v_j} \partial_{e_i}T\ar[r]^{\partial_{v_j}}  & \partial_{e_i}T \ar[r]^{\partial_{e_i}} & T}
\end{equation*}  
(Notice that after contracting $e_i$ as in the tree $\partial_{e_i}T$, the vertex $v_j$ above $e_j$ in $\partial_{e_i}T$ is a ``composed vertex'' of the vertices $v_j$ and $v_i$ originally in $T$.)

\subsection{Conilpotent copresheaves}\label{Conilp}
A copresheaf $N : \RR \to \Ch$ is said to be conilpotent if for any object $S$ in $\RR$ and any $y \in N(S)$, there is only a finite set of morphisms $\alpha :S \to T$ for which $\alpha_*(y)$ is non-zero. In other words, for any $y \in N(S)$, $\alpha_*(y)$ becomes zero for morphisms $\alpha : S \to T$ into sufficiently large trees $T$. For example, let $C$ be a cocommutative coalgebra in chain complexes, and let $N(C)$ be its nerve as defined in Section~\ref{Section:Operads}. 
Thus $N(C) : \RR \to \Ch$ is defined by
$$N(C)(S)=\underline k (S) \otimes C^{\otimes \lambda S} = C^{\otimes \lambda S}.$$
Here $\underline k : \RR \to \Ch$ is the constant functor with value $k$ representing the cooperad for cocommutative coalgebras, and $\lambda S$ is the set of leaves of the tree $S$. If $S\to T$ is an inner morphism then $N(C)(S) \to N(C)(T)$ is an isomorphism. But if $S\to T$ is an external face map, ie. a pruning, it involves a non-trivial coproduct $C^{\otimes \lambda S}\to C^{\otimes \lambda T}$. Thus $N(C)$ is a conilpotent copresheaf precisely when $C$ is a conilpotent coalgebra in the usual sense.

We will write
$$CoPsh^c(\RR) \subset CoPsh(\RR)$$
for the full subcategory of conilpotent presheaves on $\RR$.
We notice that the bar construction $\Bar M$ of a presheaf $M$ is conilpotent, so $\Bar$ is a functor $Psh(\RR) \to CoPsh^c(\RR) $.

\subsection{The cobar construction of a conilpotent presheaf}
Let $N : \RR \to \Ch$ be a conilpotent copresheaf. We will now define a cobar presheaf
$$\coBar (N) : \RRop \to \Ch$$
by the formula
$$\coBar(N)(S)=\left(\bigoplus_{\alpha: S\to T} N(T) \otimes k[en^+_T] \right)_{coinv}$$
where $\alpha$ runs over the infinite set of morphisms in $\RR$ with domain $S$, while the degree of a triple $(\alpha, y \otimes e)$ (which we also write $(\alpha,y,e)$) is $y-e$, where $e$ stands here for $p$ if $e=(e_1, \ldots, e_p)$. So we can also write
$$\coBar(N)(S)=\left(\bigoplus_{\alpha,e} N(T)[-e] \right)_{coinv}$$
The coinvariants are taken for permutations of the enumeration $e$ and for isomorphisms under $S$. So for $\sigma \in \Sigma_{p}$,
$$(\alpha, y \otimes e)=(-1)^\sigma(\alpha, y\otimes e\sigma).$$
And for $\gamma : T \isomap T'$,
$$(\alpha, y \otimes e)=(\gamma\alpha, \gamma_* y\otimes \gamma e).$$
The differential is defined by
$$\partial(\alpha, y \otimes e) = (\alpha, \partial_Y(y) \otimes e) +(-1)^y \int_d (d \circ \alpha, (\partial_d)_* y \otimes de)
- (-1)^y \int_{v \ top} (d \circ \alpha, (\partial_v)_* y \otimes de)
$$
where in the first integral $d: T\to T'$ ranges over morphisms in $\AA$ adding exactly one internal edge, and in the second integral, $v$ ranges over top vertices of $T'$ with output edge $d$ which is a leaf of $T$, as for the differential of $\Bar M$ above.
The set of such morphisms $d: T\to T'$ is infinite, and the differential only makes sense because we have assumed that $N$ is conilpotent.
As to the presheaf structure, the map $\beta^* : \coBar(N)(S) \to \coBar(N)(R)$ associated to a morphism $\beta: R \to S$ is defined by
$$\beta^*(\alpha, y \otimes e)=(\alpha \beta, y \otimes e).$$
The differential and presheaf structure are compatible with each other and are well-defined on coinvariants, so as to make $\coBar(N) : \RRop \to \Ch$ into a presheaf.

\begin{prop}\label{adjonctionAlg}
 The bar and cobar functors
 $$\coBar : CoPsh^c(\RR) \rightleftarrows Psh(\RR) : \Bar$$
 are mutually adjoint, with $\coBar$ left adjoint to $\Bar$.
\end{prop}

\begin{proof}
 The bijective correspondence  between morphisms
 $\varphi :\coBar N \to M$ and $\psi : N \to \Bar M$ is again given by the formulas 
$$\varphi_S(\alpha,y\otimes e)=(-1)^e\alpha^*\psi_T(y)_{1_T}(e)$$
 $$\psi_S(y)_\alpha(e)=(-1)^e\varphi_T(1_T, \alpha_*(y) \otimes e). $$

 Note that $\psi_S(y)$ indeed has finite support if $y$ is an element in a conilpotent copresheaf $N$. It is straightforward to check that $\varphi$ is natural if and only if $\psi$ is, that $\varphi$ preserves the differential if and only if $\psi$ does, and that the two definitions, of $\varphi$ in terms of $\psi$ and vice versa, are mutually inverse. 
 The verifications proceed exactly as for operads and cooperads, see Appendix~\ref{AppendixSigns}.
\end{proof}

 \begin{remark}
  (i) 
  Although we have no use for it in this paper, we note that the adjunction is mediated by a ``twisting cocycle'' 
  $$\tau_S : (\sbigoplus_e N(S))_{coinv} \to M(S)$$ 
  where $S$ ranges over objects of $\RR$ and $e \in en^+_S$.
  The correspondence between $\varphi$ and $\psi$ above can be expressed for $\alpha:S\to T$ and $y\in N(S)$ as 
  $$\psi_S(y)_\alpha(e)=(-1)^e \tau_T(e, \alpha_* y)$$
  $$\varphi_S (\alpha, y \otimes e)=\alpha^* \tau_T(e,y),$$
  and $\tau$ is defined in terms of $\varphi$ by this last equation, i.e.
  $\tau_S(e,y)=\varphi_S(1_S, y \otimes e)$.
  The family of morphisms $\tau_S$ for $S$ in $\RR$ is natural for isomorphisms in $\RR$, and satisfies the ``Maurer-Cartan equation''
  $$\partial_M \tau_S(e,y)-\tau_S(e, \partial_N y)=(-1)^y \int_d d^* \tau_{S'} (de,d_*y)-(-1)^y \int_{v \, top} d^* \tau_{S'} (de,(\partial_v)_*y)$$
  where both integrals range over the same $d$'s and $v$'s as before.

  (ii) We shall need an explicit description of the unit and counit of the adjunction. The unit $\eta:N \to \Bar \coBar N$ at an object $S$ is given by 
  $$\eta_S(y)_\alpha(e)=(-1)^e (1_T,\alpha_*(y) \otimes e).$$
  The counit $\varepsilon:\coBar \Bar M \to M$ is given at $S$ by
  $$\varepsilon_S(\beta, \omega \otimes e)=(-1)^e \beta^* (\omega_{1_T}(e))$$
  for $\beta : S\to T$, $\omega \in \Bar(M)(T)$ and $e\in en^+_T$.
  
  (iii) 
  For the description of $\Bar M$ in terms of coinvariants as in Section~\ref{BarAlgCoinv}, the correspondence between $\varphi:\coBar N \to M$ and 
  $\bar \psi : N \to \bBar M$ takes the form
  $$\bar \psi(y)= \int_{S\stackrel{\alpha}{\to}T,e} (-1)^e \varphi_T (1_T, e, \alpha_*y)$$
  for $y \in N(S)$. This integral (or sum over isomorphism classes, cf. Appendix of~\cite{HM21}) is finite because $N$ is assumed conilpotent. In particular, this gives the following formula for the unit 
  $\bar \eta : N \to \bBar \coBar N$ at an object $S$ of $\RR$:
  $$\bar \eta _S(y)=\int_{\alpha,e} (-1)^e (\alpha, e,(1_T,e,\alpha_*y))$$
  with $\alpha:S\to T$ and $e\in en^+_T$ as before.
  The counit $\bar \varepsilon :\coBar \bBar M \to M$ at $S$ is defined for $\alpha:S\to R$ and $e\in en^+_R$, and for any $\xi=(R \stackrel{\beta}{\to}T,f\otimes x) \in \bBar(M)(R)$, by
  $$\bar \varepsilon_S(\alpha, \xi \otimes e)=(-1)^{e+\sigma}(\beta\alpha)^*(x)$$
  provided $\beta$ is an isomorphism for which $\beta e = f \sigma$; otherwise $\bar \varepsilon_S(\alpha, \xi \otimes e)=0$.
 \end{remark}

 \section{Bar-cobar duality}
 In this section, we will discuss a version of bar-cobar (or ``Koszul'') duality for (co)presheaves on the category $\RR$, as expressed in the following theorem.
 
 \begin{theorem}\label{Th:dualite}
The bar and cobar functors $\Bar: Psh(\RR) \to CoPsh^c(\RR)$ and $\coBar : CoPsh^c(\RR) \to Psh(\RR)$ are mutually inverse up to quasi-isomorphism. More precisely, in terms of the adjunction of Proposition~\ref{adjonctionAlg}, for each conilpotent copresheaf $N$ the unit $\eta : N \to \Bar \coBar N$ is a quasi-isomorphism, as is the counit $\varepsilon : \coBar \Bar M \to M$ for each presheaf $M$.
 \end{theorem}
 
 We will prove in detail that each unit  $\eta : N \to \Bar \coBar N$ is a quasi-isomorphism, and use the description of the bar construction in terms of coinvariants to do so.
 The case of the counit is essentially dual and proceeds along the same lines, so we leave the details for the counit to the reader.
 
 \subsection{The complex $\bBar \coBar(N)$}
 
For an object $R \in \RR$, let us describe the complex $\bBar \coBar (N)(R)$ in more detail. By definition, it is the complex
$$\Big(\bigoplus_{\alpha,e} \ \bigoplus_{\beta,f} N(T)
\Big)_{coinv}$$
where $\alpha: R \to S$ and $\beta:S\to T$ are in $\RR$, $e$ runs over enumerations of the root and inner edges of $S$ and $f$ over those of $T$. Since we are taking coinvariants, we may as well assume that $e$ is the restriction of $f$ along $\beta$, and discard $e$ altogether. For similar reasons, we may assume that $S$ is a face of $T$, i.e. $\beta$ is a composition  of elementary inner and external face maps but no isomorphisms.
With these conventions, the coinvariants are only for permutations of the enumeration $f$ and for isomorphisms of $T$. In other words, elements of the complex can be written as sums of elements of the form
$$(R \stackrel{\gamma}{\to} T, S, f,y)$$
where $S \mono T$ is a face containing the image of $R$, $f$ is an enumeration of (the root and inner edges of) $T$, and $y \in N(T)$.
Taking coinvariants in particular means that for an isomorphism $\delta : T \isomap T'$, such a quadruple $(R \stackrel{\gamma}{\to} T, S, f,y)$ is identified with $(R \stackrel{\delta\gamma}{\to} T', \delta(S), \delta f,\delta_*y)$.
The complex $\bBar \coBar (N)(R)$ is the total complex of a triple complex.
The degree in this triple complex of an element $(R \stackrel{\gamma}{\to} T, S, f,y)$ is $(p,q,n)$ where $p$ is the number of vertices of $S$, $q$ that of vertices of $T$ (for which we also wrote $f$), and $n$ is the degree of $y \in N(T)$. (We also write $y$ for this degree $n$.)
The total degree is $p+(n-q)$, in accordance with the degree conventions for the bar and cobar constructions in Section~\ref{AlgBarETCobar}.
The external differential  $\partial^{ext} :\bBar \coBar (N)(R) _{p,q,n} \to \bBar \coBar (N)(R)_{p-1,q,n}$ coming from $\bBar$ is given by the formula
$$\partial^{ext}(R \stackrel{\gamma}{\to} T, S, f,y)=(-1)^{p+n-q} \sum_{i=1}^p (-1)^{i-1} (R \stackrel{\gamma}{\to} T, \hat \partial_{i} S, f,y), \quad \quad (1) $$
where $e=(e_1, \ldots, e_p)$ is the enumeration of the root and inner edges of $S$ induced by the enumeration $f$, and the sum is only over $i$ for which $\partial_{e_i} S$ contains $R$.
This differential of tridegree $(-1,0,0)$ affects neither $T$ and $f$, nor $y$.
The inner differential of $\bBar \coBar (N)(R)$ coming from the differential of $\coBar (N)(R)$ has itself an ``external'' part $\partial ^\vee$ of tridegree $(0,1,0)$ which we can write as 
$$\partial^\vee(\gamma, S, f,y)=(-1)^{n} \int_d (d\gamma, S, df,(\partial_d)_* y) -(-1)^{n} \int_{v \ top} (d\gamma, S, df,(\partial_v)_* y)
\quad \quad (2)$$
where both integrals range over the same $d$'s and $v$'s as before, and an internal part $\partial_Y$ of tridegree $(0,0,-1)$,
$$\partial_Y(\gamma, S, f,y)=(\gamma, S, f,\partial_N y). \quad \quad (3)$$

Notice that, ignoring the degree $n$ from $N$, for the bidegree $(p,q)$  we have $p \leq q$, where $p$ is bounded below by the size of the tree $R$ while $q$ is in principle unbounded above.
So the $(p,q)$-double complex lives in an octant as pictured below, where $r$ is the number of vertices of $R$.

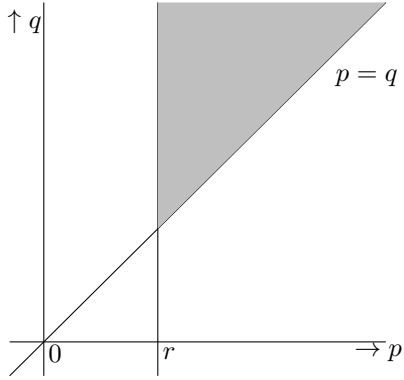
\bigskip

\begin{center}
  \begin{tikzpicture}[scale=0.05]
 \draw (10,1) -- (10,100) ;
 \draw (1,10) -- (100,10) ;
 \draw (1,1) -- (100,100) ;
 \draw (40,1) -- (40,100) ; 
 \draw (13,7) node {0} ;
 \draw (43,7) node {$r$} ;
 \draw (98,7) node {$\to p$} ;
 \draw (5,95) node {$\uparrow q$} ;
 \draw (95,80) node {$p=q$} ; 
\fill[fill=lightgray]  (40,40) -- (100,100) -- (40,100) ;
 \end{tikzpicture}
\end{center}
Let us keep the degree $n$ fixed for the moment and focus on this $(p,q)$-bicomplex.

\begin{lemma}
 Each row in the $(p,q)$-complex is acyclic, with the exception of the bottom row in bidegree $p=q=r$ which consists of $N(R)$ only (for $\gamma :R \to T$ and $S \to T$ both the identity).
\end{lemma}

\begin{proof}
 First notice that as the differential $\partial^{ext}$ of bidegree $(-1,0)$ on $(\gamma, S,f,y)$ does not affect $\gamma,f$, and $y$, for fixed $q$ (and fixed $n$ and $y$) the complex falls apart into a sum of complexes $C(\gamma,f)$ where $(\gamma,f)$ ranges over isomorphism classes of morphisms $\gamma : R \to T$ and enumerations $f\in en^+_T$. Let us fix such an isomorphism class, which we can represent by a face map $\gamma : R \to T$ and an enumeration $f$.
 Then $C(\gamma,f)$ can be identified with the tensor product of $N(T)_n$ and the complex
 $$\underline k [Fact(\gamma)]$$
 which in degree $p$ is the vector space generated by factorisations $R\mono S \mono T$ of $\gamma$ where $S \mono T$ is a face of $T$ with exactly $p$ vertices. We will simply write $S$ for such a factorisation.
 The differential is a sum
 $$\partial S = \sum \pm S'$$
 where $S'\mono S$ is an elementary face of $S$ containing $R$.
 This could be an inner face contracting an inner face $e$ of $S$, or a pruning chopping off a top vertex $v$ of $S$ not belonging to $R$.
 If $(e_1, \ldots, e_p)$ is the enumeration of inner edges of $S$ induced by the fixed enumeration $f$, then the sign $\pm$  in the previous equation is $(-1)^{p+(n-q)+i-1}$ if $S'=\partial_{e_i}S$ is an inner face, and $(-1)^{p+(n-q)+i}$ if $S'=\partial_{v_i}S$ with $v_i$ a top vertex of $S$ with output edge $e_i$. For ease of notation, let us multiply away the factor $p+(n-q)$, as it suffices to show that the complex $\underline k [Fact(\gamma)]$ with signs given by 
 $$(-1)^{i-1} \text{ if } S'=\partial_{e_i}S$$
 $$(-1)^{i} \text{ if } S'=\partial_{v_i}S$$
 is contractible.
 
 Now suppose $\gamma: R \to T$ is not an isomorphism, i.e. the complex $\underline k [Fact(\gamma)]$ is not concentrated in just one degree $p-q=r$. Let us factor  $\gamma: R \to T$ as $R \to R' \to T$ where $R\to R'$ is inner and $R' \to T$ is external. Then one of these is not an isomorphism. If $R\to R'$ isn't, choose an inner edge $a$ in $R'$ that does not belong to $R$. And if $R \to R'$ is an isomorphism, choose a leaf $a$ of $R$ (or $R'$) which is an inner edge of $T$. Such a leaf of $R$ exists because otherwise $R' \to T$ would also be an isomorphism.
 
 Using this edge $a$ of $T$, we can define a homotopy operator $h_a$ by
 $$\begin{array}{rcl}
 h_a(R\to S\to T) & = & R \to S_a \to T \text{, if $a$ is not an inner edge of $S$} \\
  & = & 0 \text{ otherwise.}
   \end{array}
   $$
 Here $S_a$ is the unique face of $T$ whose inner edges are those of $S$ together with $a$. To see that such a face $S_a$ exists (in case $a$ is not an inner edge of $S$), we again use the inner-external factorisation, as in the diagram
 \begin{equation*}
\xymatrix@M=5pt@R=12pt@C=12pt
{
R \ar[rr] \ar[dr]_{in} && S\ar[dr]_{in}  \ar[rr] && T \\ 
& S'\ar[dr]_{in} \ar[ur]_{ex} && T'\ar[ur]_{ex} \\
 && R'\ar[ur]_{ex} 
}
\end{equation*}
where edges marked $in$ are inner and those marked $ex$ are external.
 If $R\to T$ is not external, then $a$ belongs to $R'$ hence to $T'$, and $S$ is obtained from $T'$ by contracting a number of inner edges including $a$ (still assuming $a$ is not an inner edge of $S$), and $S_a$ is the tree obtained from $T'$ by contracting these inner edges except $a$.
 And if $R \to T$ is external, then $R=S'=R'$ in the diagram above, which simplifies to
 \begin{equation*}
\xymatrix@M=6pt@R=14pt@C=14pt
{
R \ar[r]^{ex} \ar[d]_{ex} & S\ar[d]^{in}\\ 
T& T'\ar[l]_{ex}
}
\end{equation*}
 where the composition $R\to T'$ is again external and $a$ is a leaf of $R$. If $a$ is not an inner edge in $S$ then $a$ is a leaf of $S$, and $S_a$ is the tree obtained by grafting the vertex of $T$ immediately above $a$ on top of $S$.
 
 To see that $h_a$ is a contracting homotopy, we may assume $a=f_1$ in the chosen enumeration $(f_1, \ldots, f_q)=f$ of  the root and inner edges of $T$. Then if $R \to T$ is not external, $S\to S_a$ is an inner face for each $S$, and
 $$\partial (h_a S)=S -h_a(\partial S).$$
 And if $R \to T$ is external, $S\to S_a$ is an external face for each $S$, and 
 $$\partial (h_a S)=- S -h_a(\partial S).$$
 In either case, we see that for $R \to T$ fixed, $h_a$ is a contracting homotopy.
 
 This shows that the complex $\underline k [Fact(\gamma)]$ is contractible if $\gamma: R \to T$ is not an isomorphism, and proves the lemma.
\end{proof}

Thus, still keeping the degree $n$ coming from the presheaf $N$ fixed, we find ourselves in the following situation.

\begin{lemma}
 Let $C_{p,q}$ be a double complex with  $C_{p,q}=0$ unless $r \leq p \leq q$ for a fixed lower bound $r$, with differentials $\partial^{ext}$ of bidegree $(-1,0)$ and $\partial^\vee$ of bidegree $(0,1)$, as in the picture (*) above.
 Suppose the rows $C_{-,q}$ are contractible for $q >r$. Then the homology of the total complex $Tot(C)_s=\bigoplus_{p-q=s} C_{p,q}$ vanishes in degree $s<0$.
\end{lemma}

\begin{proof}
The proof follows the usual tic-tac-toe argument. 
Since the grading is perhaps a bit unusual, we sketch the proof. 
Let $h:C_{p,q}\to C_{p+1,q}$ be the contracting homotopy, say 
$h \partial^{ext} +\partial^{ext}h=Id$.
(The precise signs are irrelevant here.)
Consider the total complex $Tot(C)_s=\bigoplus_{p-q=s} C_{p,q}$ with differential $D=\partial^{ext}+\partial^\vee$, and notice that it is concentrated in degrees $\leq 0$. Let $x$ be a cocycle of total degree $s<0$, say $x=(x_{p,q})_{p-q=s}=(x_{q+s,q})_{q \geq r-s}$.
As an element of the direct sum, $x_{q+s,q}$ is non-zero for only finitely many $q$'s, say $q\leq q_0$. So it suffices to prove that the homology class of $x$ can be represented by another cycle $x'=(x'_{q+s,q})_{q \leq q_0}$ for which $x'_{q+s,q}=0$ for increasingly large $q$'s. Suppose we have already achieved that $x'_{q+s,q}=0$ for $q<q'$. 
Let $x''=x'-Dh(x'_{q'+s,q'})$. Since $Dh(x'_{q'+s,q'})$ lives in bidegrees 
$(q'+s,q')$ and $(q'+s,q'+1)$, the cycle $x''$ still has the property that $x''_{q+s,q}=0$ for $q<q'$. In addition, $\partial^{ext}(x'_{q'+s,q'})=0$ so $x''_{q'+s,q'}=0$ by the homotopy equation for $h$.
This proves the lemma.
\end{proof}

\begin{remark}
 In the situation of the lemma, the homology in degree $0$ is the kernel of the map 
 $$D : \bigoplus_p C_{p,p} \to \bigoplus_p C_{p,p+1}$$
 consisting of elements $x=(x_p)_p$ with $x_p \in C_{p,p}$ such that $\partial^{ext}(x_{p+1})+\partial^\vee(x_p)=0$ in $C_{p,p+1}$. 
\end{remark}

\begin{proof}[Proof of the theorem]
 Let us go back to the triple complex
 $\bBar \coBar (N)(R)_{p,q,n}$ with its differentials
 $\partial^{ext}, \partial^\vee, \partial_N$ of degrees $(-1,0,0), (0,1,0)$ and $(0,0,-1)$, and total degree $p-q+n$.
 Let $K_n \subset \bigoplus_p \bBar \coBar (N)(R)_{p,p,n}$ be the kernel as in the previous remark.
 The differential of $N$ induces a differential $\partial_N : K_n \to K_{n-1}$, and by the previous lemma and remark, the inclusion if $(K_n, \partial_N)$ into the total complex of $\bBar \coBar (N)(R)$ is a quasi-isomorphism.
 We claim that this inclusion is exactly the image of the unit $\bar \eta: N(R) \to\bBar \coBar (N)(R)$, which would complete the proof of the theorem.
 
 To check this claim, observe that elements of $K_n$ can be represented as finite support assignments of an element
 $$y_{\gamma,e} \in N(T)_n$$
 to each $\gamma:R\to T$ and enumeration $e=(e_1, \ldots, e_p) \in en^+_T$ ; or in additive notation as a sum of elements of the form $(R \to S \to T, e, \gamma)$, as
 $$y=\sum (R\stackrel{\gamma}{\to}T=T, e, y_{\gamma,e})$$
 with the usual identification for coinvariants.
 The equation $\partial^{ext}y+\partial^\vee y=0$ means that for each composition of the form $R \stackrel{\alpha}{\to} \partial_{e_i}T \stackrel{\partial_{e_i}}{\to}T$ and enumeration $e$ as above,
 $$(-1)^{(p-1)+(n-p)} (-1)^{i-1} y_{\partial_{e_i}\circ \alpha,e} 
 = (-1)^n (\partial_{e_i})_*(y_{\alpha,e_1 \ldots \hat{e_i} \ldots e_p});$$
 or equivalently,
 $$y_{d\alpha, df}=-d_* y_{\alpha, f}$$
 for each $R \stackrel{\alpha}{\to} T \stackrel{d}{\to}T'$ where $f$ enumerates inner edges of $T$ and $d:T\to T'$ is inner.
 
 For an external face $\partial_{v_i}$ associated to the vertex $v_i$ above the edge $e_i$, the equation $\partial^{ext}y+\partial^\vee y=0$ means
 $$(-1)^{(p-1)+(n-p)} (-1)^{i} y_{\partial_{e_i}\circ \alpha,e} 
 = -(-1)^n (\partial_{v_i})_*(y_{\alpha,e_1 \ldots \hat{e_i} \ldots e_p});$$
 which again gives the equation
 $$y_{d\alpha, df}=-d_* y_{\alpha, f}$$
 for each $d:T\to T'$ external face.
 Thus by induction on the size on $T$, we obtain  $y_{\alpha,e}=(-1)^e \alpha_*(y_{id_R})$, hence $y=\bar\eta(y_{id_R})$.
\end{proof}

%
%

\section{Bar and Cobar constructions for (co)algebras}\label{Section:Alg}

Recall from Section~\ref{Section:Operads} that there are adjoint functors bar and cobar, $\Bar: Psh(\AA) \to CoPsh(\AA)$ and $\coBar : CoPsh(\AA) \to Psh(\AA)$, which induce a similar adjunction (modulo suspension) between preoperads and precooperads (respectively, $\infty$-operads and $\infty$-cooperads), cf. Theorem~\ref{OperadicTheorem}. Moreover, in each case the unit and counit are quasi-isomorphisms. In Section~\ref{AlgBarETCobar} we extended the first adjunction to an adjunction
$$\coBar : CoPsh^c(\RR) \rightleftarrows Psh(\RR): \Bar$$
between conilpotent copresheaves and presheaves on the larger category $\RR$, which again forms an equivalence of categories up to quasi-isomorphisms. The goal of this section and the next is to show that for a given preoperad $X$, this last adjunction lifts to one between prealgebras over $X$ and precoalgebras over $\Bar(sX)$.

For a preoperad $X$, the precooperad structure on $\Bar(sX)$ depends on a correspondence between pairs of morphisms $R\to U$ and $S \to V$ in $\AA$ on the one hand, and morphisms $R \circ_e S \to U \circ_e V$ in $\AA$ between graftings along a leaf $e$ of $R$ on the other, cf~\cite[Lemma 5.12]{HM21}. We begin by examining a similar correspondence for the category $\RR$ related to pre(co)algebra structures over $\Bar(sX)$. To this end, let $S$ be a fixed object of $\RR$.  Consider the category of two-fold extensions
$$ S \stackrel{\alpha}{\to} T \stackrel{\beta}{\to} U$$
where $\alpha$ is a pruning in $\RR$ and $\beta$ is an arbitrary map in $\RR$.
A morphism from one such extension to another one $ S \stackrel{\alpha'}{\to} T' \stackrel{\beta'}{\to} U'$ is a pair of isomorphisms $T \to T'$ and $U \to U'$ making the following diagram commute:
\begin{equation*}
\xymatrix@M=6pt@R=8pt{
& T \ar[r]^\beta \ar@{->}[dd]^\wr & U\ar@{->}[dd]^\wr  \\
 S\ar[dr]^{\alpha'}\ar[ur]^\alpha &&  \\
 &  T'\ar[r]^{\beta'} & U'.}
\end{equation*} 
Let us suggestively denote this category, in fact a groupoid, by
$$S / (Ext,\RR).$$
We will now give another description of this category. To this end, define for each set $L$ a groupoid $\mathbb G_L$ whose objects are $L$-indexed families of pairs of morphisms $(T_\ell \to U_\ell : \ell \in L)$ in $\RR$. The morphisms in $\mathbb G_L$ from $(T_\ell \to U_\ell)_{ \ell \in L}$ to $(T'_\ell \to U'_\ell)_{ \ell \in L}$ are $L$-indexed families of pairs of isomorphisms $T_\ell \isomap T'_\ell$ and $U_\ell \isomap U'_\ell$  making each of the squares
\begin{equation*}
\xymatrix@M=6pt@R=12pt{
 T_\ell\ar[r] \ar@{->}[d]^\wr  & U_\ell\ar@{->}[d]^\wr  \\
 T'_\ell\ar[r] & U'_\ell }
\end{equation*} 
commute.
Clearly $\mathbb G_L$ is functorial in $L$, in the sense that any function $f:K\to L$ between sets induces a functor $f^* : \mathbb G_L \to \mathbb G_K$.
Also consider the groupoid $(S / \AA)_{iso}$ whose objects are morphisms $\gamma : S\to V$ in $\AA$ and whose morphisms are isomorphisms under $S$.
This groupoid acts on $\mathbb G_-$ with fibre $\mathbb G_{\lambda(V)}$ at $\gamma: S \to V$. So we can form the semi-direct product
$$(S / \AA)_{iso} \ltimes \mathbb G_\lambda.$$
Its objects are pairs $ (S \stackrel{\gamma}{\to} V, (T_\ell \stackrel{\beta_\ell}{\to} U_\ell)_{\ell \in \lambda(V)})$ with $\gamma$ in $\AA$ and $(T_\ell \stackrel{\beta_\ell}{\to} U_\ell)_{\ell \in \lambda(V)}$ an object in $\mathbb G_{\lambda(V)}.$
Its morphisms from such an object to another one $ (S \stackrel{\gamma'}{\to} V', (T'_\ell \stackrel{\beta'_\ell}{\to} U'_\ell)_{\ell \in \lambda(V')})$ are families of isomorphisms as in
\begin{equation*}
\xymatrix@M=4pt@R=14pt{
 & S \ar@{->}[dl]_\gamma \ar@{->}[dr]^{\gamma'}&   \\
 V\ar[rr]_f^\sim && V' }
\quad \quad \quad 
\xymatrix@M=6pt@R=12pt{
 T_\ell\ar[r] \ar@{->}[d]^\wr  & U_\ell\ar@{->}[d]^\wr  \\
 T'_{f(\ell)}\ar[r] & U'_{f(\ell)} }
 \quad 
 (\ell \in \lambda(V))
\end{equation*} 

\begin{lemma}\label{L.Decomp}
 There is an equivalence of categories (groupoids)
 $$ S/ (Ext, \RR) \simeq (S / \AA)_{iso} \ltimes \mathbb G_\lambda.$$
\end{lemma}

\begin{proof}
We will just indicate how to transform an object in one of these categories to an object in the other, and vice versa, leaving further details to the reader.

First, given $S \stackrel{\alpha}{\to} T \stackrel{\beta}{\to} U$ in $ S/ (Ext, \RR)$, factor $\beta \circ \alpha$ as an inner face followed by an external one, as in 
\begin{equation}\label{lemme61}
\xymatrix@M=4pt@R=14pt{
 S\ar[r]^\alpha \ar@{-->}[d]_\gamma  & T\ar@{->}[d]^\beta  \\
 V\ar@{-->}[r]^\varepsilon & U }
\end{equation} 
The pruning $S \stackrel{\alpha}{\to} T$ exhibits $T$ as a grafting of a tree $T_\ell$ on each leaf $\ell$ of $S$, denoted 
$$T\cong S \circ (T_\ell)_{\ell \in \lambda(S)}.$$
Similarly, $U\cong V \circ (U_{\gamma(\ell)})_{\ell \in \lambda(S)}$, as $\gamma : S \to V$ in $\AA$ induces an isomorphism $\lambda(S) \to \lambda (V)$.
By commutativity of the square above, $\beta$ sends the edge $\alpha(\ell)$ in $T$ to the edge $\varepsilon \gamma(\ell)$ in $U$, so $\beta$ restricts to a map $\beta_\ell$ from the subtree $T_\ell$ of $T$ with root $\alpha(\ell)$ to the subtree $U_\ell$ with root $\varepsilon \gamma (\ell)$, and we obtain a family
$$(\beta_\ell: T_\ell \to U_\ell)_{\ell \in \lambda(S)}$$
of morphisms in $\RR$.
This family together with the map $S \stackrel{\gamma}{\to}V$ in $\AA$ defines an object in $(S / \AA)_{iso} \ltimes \mathbb G_\lambda$.

In the other direction, given $\gamma : S \to V$ in $\AA$ and such a family $(\beta_\ell: T_\ell \to U_\ell)_{\ell \in \lambda(S)}$, we can graft the $T_\ell$'s onto the leaves $\ell$ of $S$ to obtain a pruning $\alpha : S \to T$. And we can graft each tree $U_\ell$ on top of the leaf $\gamma(\ell)$ of $V$ to obtain a pruning $\varepsilon : V \to U$. The map $S \stackrel{\gamma}{\to}V$ and the family of maps $T_\ell \stackrel{\beta_\ell}{\to} U_\ell$ together define a map $\beta : T \to U$ which fits into a commutative square as the one above, and in particular defines an object $(S \stackrel{\alpha}{\to} T \stackrel{\beta}{\to} U)$ in $S/ (Ext, \RR)$. 

As said, we leave it to the reader to check that these two constructions are functorial and mutually inverse.
\end{proof}

\begin{remark}
Consider an object $S \stackrel{\alpha}{\to} T \stackrel{\beta}{\to} U$ in  $S/ (Ext, \RR)$ and the corresponding object $(\gamma,(\beta_\ell: T_\ell \to U_\ell)_{\ell \in \lambda(S)})$
in $(S / \AA)_{iso} \ltimes \mathbb G_\lambda$ as above.
Then the inner edges of $U$ are the inner edges of $V$ together with the roots and inner edges of the $U_\ell$'s. So an enumeration $e$ in $en^+_U$ corresponds (up to permutation) to an enumeration $f$ in $en_V$, a sequence of enumerations $g_\ell$ in $en^+_{U_\ell}$, and the root of $V$ (or of $S$).
\end{remark}

\begin{prop}\label{BarCoalgSurBar}
 (i) Let $X$ be a preoperad. Then for any prealgebra $M$ over $X$, the bar construction $\Bar M$ has the natural structure of a precoalgebra over $\Bar (sX)$.
 
 (ii) If in addition $X$ is an $\infty$-operad and $M$ is an $\infty$-algebra over $X$, then $\Bar M$ is an $\infty$-coalgebra over $\Bar(sX)$.
\end{prop}

The constructions involved in the proof of the proposition will ostensibly be functorial in $M$. So we obtain a commutative diagram of functors
\begin{equation}
\xymatrix@M=4pt@R=14pt{
\infty Alg(X)\ar@{^{(}->}[r] \ar[d]^\Bar & Prealg(X)\ar[r] \ar[d]^\Bar & Psh(\RR) \ar[d]^\Bar \\
\infty Coalg \Bar(sX) \ar@{^{(}->}[r]& Precoalg \Bar(sX) \ar[r]& CoPsh^c(\RR) 
}
\end{equation} 
where the horizontal functors are the inclusion and the forgetful functor, respectively.

\begin{proof}
 (i) Let $M$ be a prealgebra over $X$. The structure on $\Bar M$ of coalgebra over $\Bar (sX)$ is given by maps of degree zero
 $$\Bar M(T) \to \Bar(sX)(S) \otimes \Bar(M)(T / S),$$
 one such map for every pruning $\alpha: S \to T$.
 We will instead define maps
 $$\nabla^\alpha : \Bar M(T) \to \Bar(X)(S) \otimes \Bar(M)(T / S)$$
 of degree $-1$, and obtain the required structure maps  by composition with the suspension $\Bar X \to \Bar(sX)$.
 
 So fix a pruning $\alpha : S \to T$, and write $\Bar M(T)$, as ``invariant product''
 $$\Bar M (T) = \Big( \prodc_{\beta,e} M(U)[e] \Big)^{inv}$$
 where $\beta$ runs over morphisms $\beta: T \to U$ in $\RR$ and $e$ over enumerations in $en_U^+$.
 So elements of $\Bar M(T)$ are finitely supported functions $\omega$ with $\omega_\beta(e) \in M(U)$ of degree $\omega-e$.
 To define $\nabla^\alpha(\omega)$, we proceed as in Section~\ref{barcomplexpreoperad} and first consider the isomorphism
  $$\prodc_{\gamma,f} X(V)[f] \otimes  \prodc_{\beta_\ell,g_\ell} M(U_\ell)[g_\ell]  
  \stackrel{\mu}{\longrightarrow}  
  \prodc_{\gamma,\beta_\ell, f, g_\ell} X(V) \otimes \sbigotimes_\ell M(U_\ell) [f+\Sigma g_\ell].$$
  Here $\gamma$ ranges over morphisms $S \to V$ in $\AA$, $\ell$ over leaves of $S$, and $\beta_\ell$ over morphisms $T_\ell \to U_\ell$ in $\RR$ where $T_\ell$ is the part of $T$ above the edge $\alpha(\ell)$ as in the proof of Lemma~\ref{L.Decomp}.
  Moreover, $f$ ranges over enumerations in $en_V$ and each $g_\ell$ over those in $en^+_{U_\ell}$. With the notation thus explained, the map $\mu$ is defined by
  $$\mu(\omega, (\psi_\ell)_\ell)_{\gamma, \beta_\ell}(f, g_\ell)
  =
  (-1)^{f \cdot \Sigma g_\ell} (\omega_\gamma(f), (\psi_\ell)_{\beta_\ell}(g_\ell)_\ell).$$
  This map $\mu$ induces a similar isomorphism after taking invariants.
  
  Now we can introduce a map $\nabla'$ of degree $-1$
  $$\prodc_{\beta,e} M(U)[e] \to \prodc_{\gamma,\beta_\ell, f, g_\ell} X(V) \otimes \bigotimes_\ell M(U_\ell) [f+\Sigma g_\ell]$$
  whose value at $\zeta \in \displaystyle{\prodc_{\beta,e} M(U)[e]}$ is defined as follows.
  Take any indices in the codomain product, i.e., a map $\gamma : S \to V$ in $\AA$, maps $\beta_\ell : T_\ell \to U_\ell$ in $\RR$ and enumerations $f$ in $en_V$ and $g_\ell$ in $en^+_{U_\ell}$. Then $\gamma$ and the $\beta_\ell$'s together define an object in $(S / \AA)_{iso} \ltimes \mathbb G_\lambda$. Let $S \stackrel{\alpha}{\to} T \stackrel{\beta}{\to} U$  be the object in $S/ (Ext, \RR)$ corresponding to it under Lemma~\ref{L.Decomp}, where we have used the same notation as in the proof of the lemma,
  and let $e \in en^+_U$ be the enumeration of $U$ given by the root of $U$, the enumeration $f$ and the enumerations $g_\ell$'s in some order, cf. the previous Remark. Then define
  $$\nabla'(\zeta)_{\gamma, \beta_\ell}(f,g_\ell)=(-1)^\zeta \theta^\varepsilon (\zeta_\beta(e))$$
  where $\theta^\varepsilon : M(U) \to X(V) \otimes \sbigotimes_\ell M(U_\ell)$ is the structure map of $M$ as algebra over $X$ for the map $\varepsilon : V \to U$ in $\RR$ (cf. diagram~(\ref{lemme61}) in the proof of Lemma~\ref{L.Decomp}). A priori the values of $\nabla'(\zeta)$ depend (up to sign) on how we reconstruct the enumeration $e$ from $f$ and the $g_\ell$'s.
  But $\nabla'(\zeta)$ gives a well-defined element in the invariant product if $\zeta$ is invariant itself.
  In this way, we obtain for any given pruning $\alpha: S\to T$ as above a map of degree $-1$
 $$\nabla^\alpha : \Bar M(T) \to \Bar(X)(S) \otimes \Bar(M)(T / S).$$
  We claim that this map satisfies the required associativity and unit conditions, modulo suspension; i.e., the map 
 $(\Bar s \otimes id) \circ \nabla^\alpha : \Bar M(T) \to \Bar(X)(S) \otimes \Bar(M)(T / S)$ does. The reader may wish to check this directly. But it is notationally more convenient to do this for a description of the maps in terms of coinvariants as in the next remark. Proving that the differential is a coderivation with respect to this map $\nabla^\alpha$ is also easier in terms of coinvariants.
 
 (ii) Again consider an external face map $\alpha : S \to T$. By invariance, we can fix enumerations and use the equivalence of Lemma~\ref{L.Decomp} to write the structure map $\nabla^\alpha$ of  $\Bar M$ as
 $$\Big( \prodc_{\beta : T \to U} M(U) \Big) ^{inv} \to
  \Big( \prodc_{\beta : T \to U} X(V) \otimes M(U / V)\Big) ^{inv}$$
  where $V$ is obtained by factoring $\beta \circ \alpha : S \to U$ as in Lemma~\ref{L.Decomp}. Then the structure map becomes a product of structure maps $M(U) \to X(V) \otimes M(U / V)$ for $M$ itself for external face maps $\varepsilon : V \to U$. These are all quasi-isomorphisms if $M$ is an algebra over the $\infty$-operad $X$, hence so is $\nabla^\alpha$, proving part (ii).
\end{proof}

\begin{remark}\label{Rem64}
 The coalgebra structure map $ \Bar M(T) \to \Bar(X)(S) \otimes \Bar(M)(T / S)$ for a given external face map $\alpha : S \to T$ described above can be transfered to the bar construction in terms of coinvariants via the isomorphism $\rho : \bBar M \to \Bar M$, and it will be useful to be more explicit about this. To do so, consider the following diagram, where we have deleted the subscripts for coinvariants and superscripts for invariants in the bottom and top rows to save some space:
\begin{equation*}
\hspace{-2cm}
\xymatrix@M=4pt@C=52pt{
\displaystyle{\prodc_{\beta,e} M(U)[e] \ar[r]^{\nabla^\alpha \quad \ \quad \ \quad } \ar@/^2pc/[rr]^{\nabla'}} &
\displaystyle{\prodc_{f} X(V)[f] \otimes \sbigotimes_\ell \prodc_{g_\ell} M(U_\ell) [g_\ell]}
\ar[r]^\mu
&
\displaystyle{\prodc_{f, g_\ell} X(V) \otimes \sbigotimes_\ell M(U_\ell) [f+\Sigma g_\ell]}\\
\displaystyle{\sbigoplus_{\beta,e} M(U)[e] \ar@{-->}[r]^{\overline \nabla^\alpha  \quad \ \quad \ \quad }\ar[u]^\rho} &
 \displaystyle{\sbigoplus_{f} X(V)[f] \otimes \sbigotimes_\ell \sbigoplus_{g_\ell} M(U_\ell) [g_\ell]}
\ar[r]^{\overline \mu} \ar[u]^\rho &
\displaystyle{\sbigoplus_{f, g_\ell} X(V) \otimes \sbigotimes_\ell M(U_\ell) [f+\Sigma g_\ell]}\ar[u]^\rho
}
\end{equation*} 
(We have also omitted the indices $\gamma, \beta_\ell$ going with $f$ and $g_\ell$ to save space.)
The map $\bar \mu$ making the right hand square commute is given by
$$\bar \mu ((f,x)\otimes \sbigotimes_\ell (g_\ell, y_\ell))=(-1)^{f \cdot (\Sigma g_\ell + y_\ell)} ((f,(g_\ell)_\ell), x\otimes \sbigotimes_\ell y_\ell).$$
Using the explicit description of $\nabla'$ in the previous proof, we find the following description of the map 
$$\bar \mu \overline\nabla^\alpha :
\sbigoplus_{\beta,e} M(U)[e] \to 
\sbigoplus_{f} X(V)[f] \otimes \sbigotimes_\ell \sbigoplus_{g_\ell} M(U_\ell) [g_\ell].
$$
Fix a summand in the domain of $\overline \nabla^\alpha$, i.e., a pair consisting of a morphism $\beta : T \to U$ in $\RR$  and an enumeration $e$ in $en^+_U$. By the correspondence of Lemma~\ref{L.Decomp}, 
$S \stackrel{\alpha}{\to} T \stackrel{\beta}{\to} U$ and $e$ correspond to a morphism $\gamma : S \to V$ in $\AA$ and a morphism $\beta_\ell: T_\ell \to U_\ell$ in $\RR$ for each leaf $\ell$ of $S$ (or equivalently of $V$), and enumerations $f$ in $en_V$ and $g_\ell$ in $en^+_{U_\ell}$ such that $e=(root, f, (g_\ell)_\ell)$ after a suitable permutation.
Then 
$$\bar \mu \overline\nabla^\alpha (e,y)= (-1)^{e+y} (f, (g_\ell)_\ell, \theta^\varepsilon (y))$$
where $\varepsilon : V \to U$ as in Lemma~\ref{L.Decomp} above. The sign 
$(-1)^{e+y}$ corresponds to the sign $(-1)^\zeta$ in the definition of $\nabla'$.
Using the sign for $\bar \mu$, this gives the following formula for $\overline \nabla^\alpha$:
$$\overline\nabla^\alpha (e,y)= (-1)^{f \cdot (\Sigma g_\ell + y_\ell)+e+y}
((f,(g_\ell)_\ell), x\otimes \sbigotimes_\ell y_\ell).$$
where $e=f+\Sigma g_\ell +1$ and $\theta^\varepsilon(y)= \sum x\otimes \sbigotimes_\ell y_\ell$ (rewriting the last term as $(f,x)\otimes \sbigotimes_\ell (g_\ell, y_\ell)$ in addition introduces the usual Koszul sign).
\end{remark}

\begin{example}
 Let $P$ be a ``classical'' operad in the category $\Ch$ of chain complexes, with associated ``strict'' $\infty$-operad $NP : \AAop \to \Ch$.
 The bar complex $\Bar(sNP) : \AA \to \Ch$ defines a strict cooperad $BP$ (denoted $C_{\Bar(sNP)}$ in Section~\ref{2.7}), with
 $$BP(n)=\Big( \prod_{\alpha: C_n \to T,e} sNP(T)[e] \Big)^{inv}$$
 where $\alpha$ and $e$ range respectively over morphisms in $\AA$ and enumerations $e \in en_T$.
 Its cooperad structure maps $\nabla_i : BP(n)\to BP(k) \otimes BP(\ell)$ for $k+\ell=n+1$ and $i=1, \ldots, k$ can explicitly be described as follows.
 The codomain of $\nabla_i$ is isomorphic to an invariant product
 $$ \Big( \prod_{\beta, f, \gamma, g} sNP(R)[f] \otimes sNP(S)[g] \Big)^{inv}$$
 where the indices range over morphisms $\beta : C_k \to R$ and $\gamma : C_\ell \to S$ in $\AA$,
together with enumerations.
 For $\omega$ in $BP(n)$, the value of $\nabla_i(\omega)$ at $\beta,f, \gamma,g$ is the image of $\omega_\alpha(e)$ under the isomorphism
 $$sNP(T)[e] \to sNP(R)[f] \otimes sNP(S)[g]$$
 where $\alpha$ is the composition $C_n \to C_k \circ_i C_\ell \to R \circ_{\beta(i)} S$ and $e$ is the enumeration of $R \circ S$ given by $\beta(i), f$ and $g$.
 
 Now let $A$ be a $P$-algebra in the classical sense, with associated functor 
 $N(P,A) : \RR^{op}\to\Ch$ as in Example~\ref{2.1}. Its bar complex $\Bar N(P,A) : \RR \to \Ch$ has the structure of a coalgebra over $\Bar(sNP)$.
 Write
 $$BA=\Bar(N(P,A))(\eta)= \Big( \prodc_{U,f} NP(U)[f] \otimes A^{\otimes \lambda U} \Big)^{inv}$$
 where $U$ ranges over objects of $\RR$ and $f$ over elements in $en^+_U$.
 Then $BA$ is a coalgebra over $BP$, with structure maps
 $$BA \stackrel{\nabla}{\to} BP(n) \otimes BA^{\otimes n}$$
 described as follows.
The codomain of $\nabla$ is (isomorphic to) an invariant product over maps $\alpha : C_n\to T$ in $\AA$, trees $U_1, \ldots, U_n$ and enumerations $e\in en_T$ and $f_i \in en^+_{U_i}$. The map $\nabla$ is the one making the diagram with the projection on such a factor
\begin{equation*}
\xymatrix@M=4pt@R=16pt@C=40pt{
BA\ar@{-->}[rr]^\nabla \ar[d]^{proj_{U,e}} & & BP(n) \otimes BA^{\otimes n}\ar[d]_{proj} \\
   NP(U)  \otimes A^{\otimes \lambda U} \ar[rr] & &
  sNP(T) \otimes NP(U_1) \otimes \ldots \otimes NP(U_n) \otimes \sbigotimes_{i=1}^n A^{\otimes \lambda U_i} 
 }
\end{equation*} 
commute, where $U$ is the tree obtained by grafting $U_1, \ldots, U_n$ onto the leaves $\alpha(1), \ldots, \alpha(n)$ of $T$, and $f$ is the enumeration given by the root, $e$ and $f_1, \ldots, f_n$ together. (We have not indicated the shifts in degree in the diagram. The enumeration $f$ is one longer then $e, f_1, \ldots, f_n$ together, accounting for the suspension in the bottom morphism.)
\end{example}

We now turn to the analogue of Proposition~\ref{BarCoalgSurBar} for the cobar construction $\coBar$ adjoint to $\Bar$. The reader will observe that this cobar analogue is considerably easier than the bar case just dealt with.

\begin{prop}\label{CobarAlgsurCobar}
 (i) Let $Y$ be a precooperad. Then for any precoalgebra $N$ over $Y$, the cobar complex $\coBar N$ has the natural structure of a prealgebra over the preoperad $\coBar(s^{-1}Y).$
 
 (ii) If in addition $Y$ is an $\infty$-cooperad and $N$ is an $\infty$-coalgebra over $Y$, then $\coBar N$ is an $\infty$-algebra over $\coBar(s^{-1}Y).$
\end{prop}

\begin{proof}
 As in the case of Proposition~\ref{BarCoalgSurBar} for the bar construction, Part (ii) will be evident from the explicit description of the structure maps involved, and we will only prove Part (i).
 
 Recall that for a tree $T$,
$$\coBar(N)(T)=\Big(\bigoplus_{\beta: T\to U} N(U) \otimes k[en^+_U]\Big)_{coinv}$$
$$= \Big(\bigoplus_{\beta,e} N(U)[-e]\Big)_{coinv}$$
where in the first sum, $y \otimes e$ has degree $y-e$. We need to describe for each external face map (pruning) $S \stackrel{\alpha}{\to} T$ a map 
$$\coBar(N)(T) \to \coBar(s^{-1}Y)(S) \otimes \coBar(N)(T / S)$$
Consider a summand of $\coBar(N)(T) $ labelled by $\beta : T \to U$ and $e$.
By Lemma~\ref{L.Decomp}, the composition $s \stackrel{\alpha}{\to} T \stackrel{\beta}{\to} U$ corresponds to a morphism $\gamma: S \to V$ in $\AA$ together with prunings $\beta_\ell: T_\ell \to U_\ell$, one for each leaf $\ell$ of $S$, such that $T$ is the grafting of the $T_\ell$'s onto $S$ and $U$ is that of the $U_\ell$'s onto $V$, while we denoted the corresponding pruning by $\varepsilon : V \to U$. Moreover, the enumeration $e\in en^+_U$ can be taken to be of the form $(root, f, (g_\ell))$ where $f \in en_V$ and $g_\ell \in en^+_{U_\ell}$ (for some chosen ordering of the leaves of $S$). The precoalgebra structure of $N$ gives a morphism 
$$\theta^\varepsilon : N(U)[-e] \to s^{-1}Y(V)[-f] \otimes N (V /U)[-\Sigma g_\ell]$$
(where the desuspension compensates for the fact that $f+\Sigma g_{\ell}$ differs from $e$ by $1$). The required map is the map making each diagram of the form
\begin{equation*}
\xymatrix@M=4pt@R=16pt@C=40pt{
\coBar(N)(T)\ar@{-->}[rr] & & 
\coBar(s^{-1}Y)(S) \otimes \coBar(N)(T / S) \\
    N(U)[-e] \ar[rr]^{\theta^\varepsilon \quad \ \quad \ } \ar@{^{(}->}[u]  & &
 s^{-1}Y(V)[-f] \otimes N (V /U)[-\Sigma g_\ell] \ar@{^{(}->}[u]
 }
\end{equation*} 
commutes, where the vertical maps are inclusion of summands
(recall that $\coBar(N)(T / S)$ is the coproduct
$\displaystyle{\sbigoplus_{\beta_\ell, g_\ell} \sbigotimes_\ell N(U_\ell) [-g_\ell]}$).
The compatibility of this map with the differential is induced by the compatibility of the structure map $\theta$ with the differential.
\end{proof}

\section{Duality for algebras over $\infty$-operads}

We will now turn to the problem of lifting the adjunction of Theorem~\ref{OperadicTheorem} to prealgebras and precoalgebras, as stated in Theorem~\ref{ThAdjAlgCoalg} below.

To prepare for the proof of this theorem, let $X$ be a preoperad, and let $M$ be an $X$-prealgebra and $N$ be a $\Bar(sX)$-coalgebra.
Consider a morphism of presheaves $\varphi :\coBar N \to M$ over $\RR$. We wish to simplify the condition for $\varphi$ to be a map of prealgebras over $ \coBar(s^{-1} \Bar(sX)) \stackrel{\varepsilon}{\to} X$, 
or in other words, for $\varepsilon_! \coBar N \to M$ to be a map of prealgebras over $X$.
To do so, recall that for a tree $T$, we have
$$\coBar(N)(T)=\Big(\bigoplus_{\beta: T\to U,e} N(U)[-e]\Big)_{coinv}.$$
We will write
$$j_{\beta,e}: N(U)[-e] \to \coBar N (T)$$
for the canonical map. The structure of $\varepsilon_! \coBar N $ as an $X$-algebra gives for each pruning $\alpha : S\to T$ and each $\beta :T \to U$ and $e\in en^+_U$ a map
$$\nabla^\alpha_{\beta,e} = \nabla^\alpha \circ j_{\beta,e}: N(U)[-e] \to
X(S) \otimes \sbigotimes_\ell N(U_\ell) [-g_\ell] $$
with $U_\ell$ and $g_\ell$ as before, cf. Lemma~\ref{L.Decomp} and the remark following it.
In particular, for $\beta$ the identity, we have maps
$$\nabla^\alpha_{1,e}: N(T)[-e] \to
X(S) \otimes \sbigotimes_\ell N(T_\ell) [-g_\ell]$$
and the maps $\beta_\ell$ are also identities.

\begin{lemma}
 A map $\varphi :\coBar N \to M$ of presheaves over $\RR$ defines a map of $X$-prealgebras $\varepsilon_! \coBar N \to M$ if and only if for each pruning $\alpha : S \to T$ and each enumeration $e\in en^+_T$, the diagram
 \begin{equation*}
\xymatrix@M=4pt@R=16pt@C=40pt{
N(T)[-e]\ar@{->}[rr]^{\nabla^\alpha_{1,e} \quad \ \quad \ } \ar[d]^{\varphi_{T, 1, e}} & & 
 X(S) \otimes \sbigotimes_\ell N(T_\ell) [-g_\ell] \ar[d]^{1 \otimes \sbigotimes \varphi_{T_\ell, 1, g_\ell}} \\
    M(T) \ar[rr]^{\theta^\alpha \quad \ \quad \ }   & &
  X(S) \otimes \sbigotimes_\ell M(T_\ell)
 }
\end{equation*} 
 commutes, where $\varphi_{T,1,e}=\varphi_T \circ j_{1,e}$ and similarly for $\varphi_{T_\ell, 1, g_\ell}$.
\end{lemma}

\begin{proof}

 Let $\varphi :\coBar N \to M$ be such a map of presheaves. Then $\varphi$ is a map of $X$-prealgebras if and only if for each pruning $\alpha : S \to T$ and each $\beta:T\to U$ and each enumeration $e\in en^+_T$, the diagram
 \begin{equation*}\tag{$\square$}
\xymatrix@M=4pt@R=16pt@C=40pt{
N(U)[-e]\ar@{->}[rr]^{\nabla^\alpha_{\beta,e} \quad \ \quad \ } \ar[d]^{\varphi_{T, \beta, e}} & & 
 X(S) \otimes \sbigotimes_\ell N(U_\ell) [-g_\ell] \ar[d]^{1 \otimes \sbigotimes \varphi_{T_\ell, \beta_\ell, g_\ell}} \\
    M(T) \ar[rr]^{\theta^\alpha \quad \ \quad \ }   & &
  X(S) \otimes \sbigotimes_\ell M(T_\ell)
 }
\end{equation*} 
 commutes, where $\beta_\ell : T_\ell \to U_\ell$ and $g_\ell \in en^+_{U_\ell}$ as before, and where 
$\varphi_{T,\beta,e}=\varphi_T \circ j_{\beta,e}$ and similarly for $\varphi_{T_\ell, \beta_\ell, g_\ell}$.

The lemma asserts that it is enough to require this for each $\alpha: S\to T$ and each $e$ where $\beta$ is the identity.
This follows from the naturality of $\varphi$ together with the fact that for an arbitrary $\beta : T \to U$, the presheaf structure $\beta^* : \coBar N(U) \to \coBar N (T)$ sends the summand for the identity to that for $\beta$, as in 
 \begin{equation*}
\xymatrix@M=4pt@R=16pt@C=40pt{
\coBar(N)(U)\ar[rr]^{\beta^*} & & \coBar(N)(T) \\
& N(U)[-e] \ar[ul]^{j_{1,e}}\ar[ur]_{j_{\beta,e}} &
 }
\end{equation*} 
Thus, by naturality of $\varphi$, we find for $\varphi_{T, \beta, e}: N(U)[-e] \to M(T)$ that 
$$\varphi_{T, \beta, e}= \varphi_{T} \circ j_{\beta,e} = \varphi_T \circ \beta^* \circ j_{1,e} = \beta^* \circ \varphi_U \circ j_{1,e} = \beta^* \circ \varphi_{U,1,e}.$$
In other words, the left hand square in the cube below commutes, and for the same reason so does the right hand  square.
(In the cube, we delete the shifts in degree and write $\varphi_{T,\beta}$ instead of $\varphi_{T,\beta,e}$ to save some space.)
 \[
  \xymatrix{
& N(U) \ar[rr]^{\nabla^\alpha_{\beta,e}} \ar'[d][dd]^{\varphi_{T,\beta}}&& \makebox[\widthof{$B$}][l]{$X(S)\otimes N(U/S)$} \ar@{->}[dd] \\
N(U) \ar@{=}[ru] \ar[rr]^{\nabla^{\alpha \beta}_{1,e}} \ar[dd]_{\varphi_{U,1}} && X(S)\otimes N(U/S) \ar@{=}[ru] \ar[dd]^(.25){1\otimes \sbigotimes \varphi_{U_{\ell,1}}} & \\
& M(T) \ar'[r][rr]^{\theta^\alpha} && \makebox[\widthof{$B$}][l]{$X(S)\otimes M(T/S)$} \\ M(U) \ar@{>->}[ru]^{\beta^*} \ar[rr]^{\theta^{\alpha \beta}} && X(S)\otimes M(U/S) \ar@{->}[ru] &
}
\]
The front square is a square of the form ($\square$) for $\beta$ the identity, the bottom square commutes by naturality of the operadic decomposition map, and the commutativity of the back square follows from that for the front one. This proves the lemma.
\end{proof}

Now let us turn to the dual case and consider for precoalgebras $M$ over $X$ and $N$ over $\Bar(sX)$ as before a map of copresheaves $\psi : N \to \Bar M$.
We wish to examine when $\psi$ is a morphism of precoalgebras.
Recall first that for a tree $S$, the value of $\Bar(sX)$ at $S$ is given by
$$\Bar(sX)(S)=\Big( \prod_{\gamma: S\to V,f} sX(V)[f]\Big)^{inv}.$$
where $\gamma$ ranges over morphisms in $\AA$ and $f \in en_V$. Write
$$p_{\gamma,f} : \Bar(sX)(S)\to sX(V)[f] $$
for the projection. Similarly,
$$\Bar(M)(T)= \Big(\prodc_{\beta:T\to U,e} M(U)[e]\Big)^{inv}$$
and we will write 
$$q_{\beta,e} : \Bar(M)(T) \to M(U)[e]$$
for the projection.
The structure of $\Bar M$ as a $\Bar(sX)$-precoalgebra is given by maps
$$\nabla^\alpha : \Bar M(T) \to \Bar(sX)(S) \otimes \Bar(M)(T/S),$$
one for each pruning $\alpha: S\to T$. Let us write
$$\nabla^\alpha_{\beta,e}=(p\otimes q)\circ \nabla^{\alpha} : 
\Bar M(T) \to sX(V)[f] \otimes \sbigotimes_\ell M(U_\ell / T_\ell)[g_\ell]
$$
where $p=p_{\gamma,f}$ and $q=\otimes_\ell q_{\beta_\ell, g_\ell}$ for $\gamma, \beta_\ell$ and $f,g_\ell$ as in Lemma~\ref{L.Decomp} and the remark following it. The map $\psi : N \to \Bar M$ of copresheaves then gives the following commutative diagram, for any enumeration $e \in en_T^+$ and resulting enumerations $f\in en_S$ and $g_\ell \in en^+_{T_\ell}$ as before:
 \begin{equation*}
\xymatrix@M=4pt@R=16pt@C=50pt{
N(T)\ar@{->}[r]^{\tau^\alpha \quad \ \quad} \ar[d]_{\psi_T} &  
 \Bar(sX)(S) \otimes N(T/S) \ar[r]^{p_{1,e} \otimes 1} \ar[d]^{1\otimes \psi_{T/S}} &
 sX(S)[f] \otimes \sbigotimes_\ell N(T_\ell)  \ar[d]^{1 \otimes \sbigotimes \psi_{T_\ell, \beta_\ell, g_\ell}} \\
   \Bar M(T)[e]\ar[d]_{q_{1,e}} \ar[r]^{\nabla^\alpha \quad \ \quad  }   &
    \Bar(sX)(S) \otimes \Bar M(T/S) \ar[r]^{p_{1,e} \otimes 1} \ar[d]^{p \otimes q}  &
  sX(S)[f] \otimes \sbigotimes_\ell \Bar M(T_\ell) \ar[dl]^{1 \otimes \sbigotimes q_{1, g_\ell}}  \\
  M(T)[e]  \ar[r]^{\theta^\alpha \quad \ \quad  }& 
  sX(S)[f] \otimes \sbigotimes_\ell M(T_\ell)[g_\ell]  &
 }
\end{equation*} 
%
%
where $p \otimes q$ in the middle stands for $p_{1,f} \otimes \sbigotimes_\ell q_{1,g_\ell}$.

Let us write $\psi_{T, \beta,e} : N(T) \to M(U)[e]$ for $q_{\beta,e} \circ \psi_T$, where $\beta : T \to U$ and $e \in en^+_U$. Then exactly as for the previous lemma, the naturality of $\psi$ now implies the following.

\begin{lemma}
 With the notation as above, a map $\psi : N \to \Bar M$ of copresheaves over $\RR$ is a morphism of $\Bar(sX)$-algebras if and only if for each pruning $\alpha : S \to T$, the diagram 
 \begin{equation*}
\xymatrix@M=4pt@R=16pt@C=40pt{
N(T)\ar@{->}[rr]^{\nabla^\alpha_{1,e} \quad \ \quad \ } \ar[d]^{\psi_{T, 1, g}} & & 
 sX(S)[f] \otimes \sbigotimes_\ell N(T_\ell) \ar[d]^{1 \otimes \sbigotimes \psi_{T_\ell, 1, g_\ell}} \\
    M(T)[e] \ar[rr]^{\theta^\alpha \quad \ \quad \ }   & &
  sX(S)[f] \otimes \sbigotimes_\ell M(T_\ell)[g_\ell]
 }
\end{equation*} 
commutes.
\end{lemma}

As said, this lemma is proved in exactly the same way as the previous one, and we omit the details. The main result of this section now readily follows:

\begin{theorem}\label{ThAdjAlgCoalg}
 (i) Let $X$ be a preoperad. The the bar-cobar adjunction for presheaves over $\RR$
 lifts to an adjunction between prealgebras over $X$ and precoalgebras over $\Bar(sX)$,
 $$\varepsilon_! \coBar : Precoalg( \Bar sX) \rightleftarrows Prealg(X) : \Bar$$
 where $\varepsilon : \coBar (s^{-1} \Bar(sX)) \to X$ is the counit of the adjunction.
 
 (ii) The unit and counit of the adjunction are quasi-isomorphisms.
 
 (iii) If in addition $X$ is an $\infty$-operad, the adjunction restricts to an adjunction between $\infty$-algebras over $X$ and $\infty$-coalgebras over $\Bar(sX)$.
\end{theorem}

\begin{proof}
 For a prealgebra $M$ over $X$, the precoalgebra structure on $\Bar M$ established in Proposition~\ref{BarCoalgSurBar} is clearly functorial in $M$, so the bar construction lifts to a functor $Prealg(X) \to Precoalg(\Bar sX)$. Similarly, for a precooperad $Y$, Proposition~\ref{CobarAlgsurCobar} gives that the functor $\coBar$ lifts to a functor $Precoalg(Y) \to Prealg(\coBar s^{-1}Y)$. Thus we obtain a pair of functors
 $$\varepsilon_! \coBar : Precoalg( \Bar sX) \rightleftarrows Prealg(X) : \Bar$$
 as stated in the theorem.
 If one now compares the previous two lemmas, one sees that for a prealgebra $M$ over $X$ and a precoalgebra $N$ over $\Bar(sX)$, a map of copresheaves $\varphi : N \to \Bar M$ is a map of precoalgebras if and only if its transpose $\psi : \coBar N \to M$ along the adjunction of Theorem~\ref{adjonctionAlg} is a map of prealgebras. This proves Part (i). Part (ii) follows from Part (i) and Theorem~\ref{adjonctionAlg}. Finally, for Part (iii), first notice that since the counit $\varepsilon$ is a quasi-isomorphism (see Theorem~\ref{Th:dualite}), the functor $\varepsilon_!$ sends  $\infty$-algebras to $\infty$-algebras (see Section~\ref{pushforward}). Part (iii) now formally follows from Proposition~\ref{BarCoalgSurBar} and Proposition~\ref{CobarAlgsurCobar}.
\end{proof}

%
%

\section {The homology of the bar complex}

%

In our earlier paper~\cite{HM21}, we defined the \emph{dendroidal homology} $DH_*(X)$ of an $\infty$-operad $X$ as the homology of the total complex of the double complex 
\begin{equation}\label{Eq8.1}
DC_p(X)_q= \sbigoplus_\ell \bBar(X_q)(C_\ell)_p 
\end{equation}
where $C_\ell$ is the corolla with $\ell \geq 2$ leaves. 
This homology splits as a direct sum
\begin{equation}\label{Eq8.2}
DH_*(X)=\sbigoplus DH_*^{(\ell)}(X) 
\end{equation}
and the $DH_*^{(\ell)}(X)$ form a (strict) cooperad in vector spaces after suspension (i.e. the $DH_*^{(\ell)}(sX)$ do).
Indeed, $\bBar(sX)$ is an $\infty$-cooperad by  Theorem~\ref{OperadicTheorem}, so the 
$\bBar(sX)(C_\ell)$ together form a strict cooperad in chain complexes, cf.~\ref{2.7}.

We can define in an analogous way the \emph{dendroidal homology} $DH_*(M)$ of an $\infty$-algebra $M$ over an $\infty$-operad $X$ as the homology of the bar complex evaluated at the tree $\eta$,
\begin{equation}\label{Eq8.3}
DH_*(M)=H_*(\bBar(M)(\eta)). 
\end{equation}
Then $DH_*(M)$ is a coalgebra over the (strict) cooperad $DH_*(sX)$, since $\bBar(M)$ is an $\infty$-coalgebra over $\bBar(sX)$ by Theorem~\ref{BarCoalgSurBar}, whence $\bBar(M)(\eta)$ is a strict coalgebra over the cooperad given by the $\bBar(sX)(C_\ell)$'s, cf. Remark~\ref{Strictification}.

This dendroidal homology as in (\ref{Eq8.3}) of course satisfies the usual standard properties, such as invariance under quasi-isomorphism, long exact sequences and spectral sequences associated to the double complex coming from the grading in $M$ and the one by edges in a tree.

If $P$ is a classical Koszul operad in vector spaces (viewed as chain complexes concentrated in degree $0$) then $DH_*(NP)$ as defined in (\ref{Eq8.1}) agrees with the homology of an operad as defined in~\cite{GK}, and the cooperad $\{DH_*^{(\ell)}(NP)\}_\ell$ is the Koszul dual $P^\acc$ of $P$. The vector space $DH_0(NP)$ is the space of indecomposables of $P$. Similarly, if $A$ is a $P$-algebra (again concentrated in degree $0$) then $DH_*(N(P,A))$ as just defined agrees with the homology $H_*^P(A)$ of $A$ as defined in~\cite{LV}, and $DH_0 N(P,A)$ is the space of indecomposables of $A$, i.e. the cokernel of the map $$\sbigoplus_\ell P(\ell) \otimes A^{\otimes \ell} \to A.$$

%
We will now show that this homology $H_*\Bar(M)$ can be interpreted in terms of the homology of categories, as stated in Theorem~\ref{Prop:HomologyOfCat}.
More generally we will describe the value of the homology of $\Bar(M)$ at each tree.

Recall that for a small category $\CC$ and a presheaf $A: \CC^{op} \to \Ch$ with values in chain complexes, the homology $H_*(\CC,A)$ is that of the total complex of the double complex $C_*(\CC,A)$ defined by 
$$C_p(\CC,A)_q = \bigoplus_{c_0 \to \ldots \to c_p} A(c_p)_q$$
with differential in $p$ induced by the face maps $d_i:(c_0 \to \ldots \to c_p) \mapsto (c_0 \to \ldots \hat{c_i} \to \ldots \to c_p)$ of the nerve of $\CC$ and the action of $\CC$ on $A$ for $d_p$. If $\CC' \subset \CC$ is a subcategory, then if we also write $A$ for the restriction to $\CC'$, we obtain a double complex $C_*(\CC', \CC ; A)_*$ as the usual cokernel
$$0 \to C_*(\CC', A)_* \to C_*(\CC , A)_* \to C_*(\CC', \CC ; A)_* \to 0.$$
The total complex of this cokernel computes the homology $H_*(\CC', \CC ; A)_*$ of a pair of categories, the category $\CC$ and a subcategory $\CC'$.


Recall also that there is an isomorphism of graded abelian groups
$$DH_*^{(\ell)}(X)=H_*(C_\ell / \AA, C_\ell /\!/ \AA ; X)$$
(see~\cite{HM21}, but notice that our grading of $\bBar(X)$ differs by $1$ from the one in~\cite{HM21}).
Notice that there is no a priori natural (shifted) cooperad structure on the right hand side of Equation~(\ref{Eq8.2}).

Now let us return to the category $\RR$ of trees defined in Section~\ref{S:Categories}. For an object $S$ of $\RR$, the comma (or ``slice'') category of objects under $S$ is denoted $S/\RR$. Its full subcategory on the non-isomorphisms $S \to T$ is denoted $S/\!/\RR$. We then have the following property of the bar construction.

\begin{theorem}\label{Prop:HomologyOfCat}
 For any presheaf $M$ on $\RR$ with values in chain complexes, and for any object $S$ of $\RR$, there is an isomorphism
 $$H_*(\Bar(M)(S))=H_{*-p}(S/\RR, S/\!/\RR ; M),$$
 where $p$ is the number of vertices of the tree $S$.
\end{theorem}

The proof of the theorem uses the following lemma.

\begin{lemma}
 Let $R$ be an object of $\RR$, and consider the associated representable presheaf $M= \underline k \RR(-, R)$ (viewed as a presheaf of chain complexes concentrated in degree zero). Then for any object $S$ in $\RR$
 
 $$H_n(\Bar(M)(S))=
 \left\{ \begin{array}{ll}
 \underline k[Iso(S,R)] & \text{ if } n=p \\
 0 & \text{ otherwise}
 \end{array}\right.$$
 where $p$ is the number of vertices in $R$.
\end{lemma}

\begin{proof}
 It will be convenient to work with the description of $\Bar(M)$ in terms of coinvariants. So for $M=\underline k \RR (-,R)$,
\begin{eqnarray*}
\bBar(M)(S)&=& \Big(\bigoplus_{\alpha:S \to T} k [en^+_T] \otimes \underline k[Hom_\RR(T,R)]\Big)_{coinv}\\
& \cong& \Big(\bigoplus_{S \to T \to R} k [en^+_T] \Big)_{coinv}.
 \end{eqnarray*} 
Elements of $\bBar(M)(S)$ are represented by $k$-linear combinations of triples of the form $(S \stackrel{\alpha}{\to} T \stackrel{\beta}{\to} R,e)$ where $e=(e_1, \ldots, e_p)$ is an enumeration of $T$ (i.e., of the inner edges and the root edge of $T$, as before).
The differential is given up to sign by
\begin{eqnarray*}
\partial(S \stackrel{\alpha}{\to} T \stackrel{\beta}{\to} R,e)= &
\sum (-1)^i (S \to \partial_{e_i}T \to R, (e_1, \ldots, \hat{e_i}, \ldots, e_p)) \\
& - \sum (-1)^j (S \to \partial_{v_j}T \to R,(e_1, \ldots, \hat{e_j}, \ldots, e_p)) 
\end{eqnarray*}
where the first sum runs only over inner edges of $T$ that do not belong to $S$, and the second one runs over the vertices $v_j$ immediately above $e_j$ that do not belong to $S$ and lie on top of $T$ (so that the face $\partial_{v_j} T \to T$ obtained by chopping off $v_j$ is well defined).
Note that in the latter case, although $v_j$ should not belong to $S$, the edge $e_j$ can be a leaf of $S$.

Let us observe first that the composition $\beta \circ \alpha : S \to R$ does not change under the differential. So the complex $\bBar(M)(S)$ decomposes as a sum over morphisms $S \stackrel{\gamma}{\to} R$, as
$$\bBar(M)(S)=\bigoplus_\gamma \bBar(M)(S)_\gamma$$
where $\bBar(M)(S)_\gamma$ is the complex of (coinvariance classes of) triples $(\alpha, \beta, e)$ where $\beta\circ \alpha=\gamma$.
If $\gamma$ is an isomorphism, this complex is concentrated in degree $p$ and can be taken to consist of triples  $(\alpha, \beta, e)$ where $\alpha$ and $\beta$ are the identity and $\gamma$ respectively, and $e$ is a fixed enumeration. So for an isomorphism $\gamma$,
 $$H_n(\Bar(M)(S)_\gamma)=
 \left\{ \begin{array}{ll}
 \underline k & \text{ if } n=p \\
 0 & \text{ otherwise}
 \end{array}\right.$$
 We claim that the summand $\bBar(M)(S)_\gamma$  is acyclic if $\gamma$ is not an isomorphism. This would show that 
 $$H_n(\Bar(M)(S))=\underline k[Iso(S,R)] \text{ if } n=p$$
 and zero for other $n$, and prove the proposition.

 To prove the claim, consider a non-isomorphism $\gamma : S \to R$ and factor it as an inner face map $S\to R'$ followed by an external face map $R' \to R$. Then at least one of these is not an isomorphism, and we consider these two cases separately.
 
 If $S\to R'$ is not an isomorphism, fix an inner edge $a$ in $R'$ that does not belong to $S$, and define
 $$h_a(S \stackrel{\alpha}{\to} T \stackrel{\beta}{\to} R,e)=
 \left\{ \begin{array}{ll}
 (-1) ^{q+1} (S \to T_a \to R, ea) & \text{ if } a \text{ is not an edge of }T \\
 0 & \text{ otherwise}
 \end{array}\right.$$
 Here $e=(e_1, \ldots, e_q)$ and $ea$ is the enumeration putting $a$ at the end. The tree $T_a$ is the face of $R$ described as follows.
 If we factor $\alpha$ and $\beta$ as an inner face followed by an external one, as in
 \begin{equation*}
\xymatrix@M=5pt@R=12pt@C=12pt
{
S \ar[rr] \ar[dr]_{in} && T\ar[dr]_{in}  \ar[rr] && R \\ 
& S'\ar[dr]_{in} \ar[ur]_{ex} && T'\ar[ur]_{ex} \\
 && R'\ar[ur]_{ex} 
}
\end{equation*}
(where $in$ denotes an inner face and $ex$ an external face), then $a$ is an inner edge of $R'$ (by definition), hence also of $T'$. The tree $T$ is obtained by contracting some inner edges of $T'$. Now $T_a$ is obtained by contracting those edges of $T'$ \emph{except} $a$.
One checks that
$$\partial h_a(S \to T \to R,e)= -h_a \partial (S \to T \to R,e) + (S \to T \to R,e)$$
in case $a$ does not belong to $T$. In case $a$ belongs to $T$, we have
$\partial h_a(S \to T \to R,e)=0$. On the other hand, if we write $e=e'a$ putting $a=e_q$ at the end of the sequence $e=(e_1, \ldots, e_q)$ (which we may do by coinvariance), then
$$ h_a \partial (S \to T \to R,e) = (-1)^q h_a ( S \to \partial_a T \to R,e') = (S \to T \to R, e'a)$$
so again $\partial h_a(S \to T \to R,e)= -h_a \partial (S \to T \to R,e) + (S \to T \to R,e)$.
This shows that $h_a$ is a contracting homotopy, ie. $h_a \partial + \partial h_a=Id$.

Now suppose that in the factorisation $S \to R' \to R$ the first map is an isomorphism, so $\gamma : S \to R$ is an external face. Then in any factorisation $S \stackrel{\alpha}{\to} T \stackrel{\beta}{\to} R$ of $\gamma$, the map $\alpha$ is external as well. The claim in the particular case $\eta \to \eta \to C_\ell$ can be directly checked by hand.
Otherwise fix a leaf $a$ of $S$ which is an inner edge of $R$.
Such a leaf $a$ must exist as $S\to R$ is not an isomorphism and $R$ has at least one inner edge.
For a factorisation $S \stackrel{\alpha}{\to} T \stackrel{\beta}{\to} R$ of $\gamma$ where $a$ is not an inner edge of $T$, the edge is a leaf of $T$ and we can define a tree $T_a$ by grafting the vertex of $R$ immediately above $a$ onto $T$. We then define
$$h_a(S \stackrel{\alpha}{\to} T \stackrel{\beta}{\to} R,e)=
 \left\{ \begin{array}{ll}
 (-1) ^{q+1} (S \to T_a \to R, ea) & \text{ if } a \text{ is not an inner edge of }T \\
 0 & \text{ otherwise}
 \end{array}\right.$$
 where $e=(e_1, \ldots, e_q)$ and we assume $e_q=a$ in case $a$ is inner in $T$, as before. Then one checks again that 
$$\partial h_a(S \to T \to R,e)= -h_a \partial (S \to T \to R,e) + (S \to T \to R,e),$$
so $h_a$ is a contracting homotopy.

This concludes the verification of the claim, and proves the lemma.
\end{proof}

\begin{proof}[Proof of the theorem]
 By the usual double complex argument, it  suffices to prove the proposition in case $M$ is just a presheaf with values in abelian groups, viewed as chain complexes concentrated in degree zero.

 Take a presheaf $M$ of abelian groups. Fix an object $S$ in $\RR$. Then we can form a double complex $C_{*,*}$ which in bidegree $p,q$ is defined as
 $$C_{p,q}=\Big( \bigoplus_{S \to T \to R_0 \to \ldots \to R_q} k[en^+_T] \otimes M(R_q) \Big)_{coinv}.$$
 Here $S \to T \to R_0 \to \ldots \to R_q$ is a sequence of morphisms in $\RR$, where $T$ has $p$ vertices. Coinvariants are for isomorphisms of $T$ as in 
\begin{equation*}
\xymatrix@M=6pt@R=8pt{
 & T \ar[dr] \ar@{-->}[dd]^\wr & \\
 S\ar[dr]\ar[ur] && R_0 \\
 &  T'\ar[ur] &}
\end{equation*}  
 and permutations of the enumerations of $T$. The differential in $q$ is like that for the homology of $\RR$, an alternating sum of operations deleting the $R_i$. The differential in $p$ is defined by an alternating sum of faces of $T$ as in the definition of $\bBar(M)$.
 
 Now for a fixed $p$, the complex $C_{p,-}$ is a sum over isomorphism classes of arrows $S\to T$ and enumerations $e$ of complexes computing the homology of $T/\RR$ with coefficients in $M$. Since $T/\RR$ has an initial object, this homology vanishes in positive degrees, and we find
 $$H_*(C_{p,-})= \bigoplus_{S\to T,e} M(T)_{coinv}= \bBar(M)(S)_p.$$
 We conclude that the double complex computes the homology $H_*(\bBar(M)(S))$.
 
 On the other hand, for a fixed $q$, and before taking coinvariants, $C_{-,q}$ is a sum over sequences $R_0 \to \ldots \to R_q$ of complexes $\bBar(\RR(-,R_0))(S) \otimes M(R_q)$. If we write $p$ for the number of vertices in $S$, the homology of $\bBar(\RR(-,R_0))(S)$ is concentrated in degree $p$ by the lemma, and we find that 
 $$H_p(C_{-,q})= \underline k [Iso(S,R_0)] \otimes M(R_q)$$
 (since up to coinvariants there is only one factorisation $S\to T \to R$ of each isomorphism $S \to R$.) This shows that the double complex computes the homology $H_*(S/\RR, S//\RR ; M)$ shifted by $p$, and the theorem is proved.
\end{proof}

\appendix

\section{Comonads}\label{comonads}

In this appendix we will show that preoperads and their algebras are coalgebras for a comonad. This will in particular imply the existence of a change-of-base functor, right adjoint to the push-forward of prealgebras along a morphism of preoperads, introduced in Section~\ref{pushforward}.

\begin{prop}
 The category of preoperads is the category of coalgebras over a comonad 
 $$G: Psh(\AA) \to Psh(\AA).$$
 (The functor $G$ and its comonad structure will be described explicitly in the proof.)
\end{prop}

\begin{proof}
 We will give explicit descriptions of the structures involved, but leave some straightforward verifications to the reader.
 
 First of all, for a presheaf $X$ and a tree $T$, define
 $$G(X)(T)= \Big( \displaystyle{ \prod_{\alpha:S\to T} \sbigotimes_{v \in S} X (T(\alpha)_v)} \Big)^{inv}.$$
 Here $\alpha$ ranges over morphisms in $\AA$ with codomain $T$, $v$ over vertices in $S$, and $T(\alpha)_v$ is the subtree of $T$ defined as the blow-up of the vertex $v$, cf. Section~\ref{S:Categories} above. The invariants are taken over isomorphisms over $T$. To be explicit, let us write an element $\omega$ of $G(X)(T)$ as a family $(\omega_\alpha)$ where $\alpha: S\to T$, and $\omega_\alpha \in \sbigotimes_v X (T(\alpha)_v)$. Then invariance means that for each isomorphism $\gamma : S \to R$ over $T$ and the induced isomorphism $\gamma_v$ as in 
\begin{equation*}
\xymatrix@M=6pt{
 S\ar[dr]_{\alpha}\ar[rr]^{\gamma} && R \ar[dl]^{\beta}\\
 &  T  \\}
 \quad \ \quad \ \quad \
 \gamma_v :T(\alpha)_v \isomap T(\beta)_{\gamma(v)}
\end{equation*}  
 we have that 
 $$\sbigotimes_v \gamma_v^* : \sbigotimes_v X(T(\beta)_{\gamma(v)}) \to \sbigotimes_v X(T(\alpha)_{v})$$
 maps $\omega_\beta$ to $\omega_\alpha$.
 
 There are now a few things to check, all relatively straightforward, and listed below:

 (i) $G(X)$ indeed has the structure of a presheaf on $\AA$. Consider a morphism $\beta: T\to T'$ in $\AA$. 
  Then $\beta^*:G(X)(T')\to G(X)(T)$ is the map making the following diagram commute, for each $\alpha: S \to T$ in $\AA$
\begin{equation*}
\xymatrix@M=6pt{
 G(X)(T') \ar[d]_{proj_{\beta\alpha}}\ar[r]^{\beta^*} & G(X)(T) \ar[d]^{proj_{\alpha}}\\
  \sbigotimes_v X(T'(\beta\alpha)_{v}) \ar[r] &  \sbigotimes_v X(T(\alpha)_{v})\\}
\end{equation*}  
The lower arrow here is defined for each $\alpha$ by the effect of the presheaf structure of $X$ on the maps 
$$T(\alpha)_v \to T'(\beta\alpha)_v$$
for each $v \in S$.
These maps come from the successive factorisations as in
\begin{equation*}
\xymatrix@M=6pt
{S \ar[r]^\alpha &T \ar[r]^\beta & T' \\
  C_v \ar@{-->}[r] \ar[u]^{i_v} 
  &T(\alpha)_v \ar@{-->}[r]\ar@{-->}[u]^{\gamma}
  & T'(\beta\alpha)_v \ar@{-->}[u]^{\delta}
  }
\end{equation*}  
where the left hand square is the unique factorisation of $\alpha \circ i_v$ into a morphism of $\AA$ followed by an external face $\gamma$, and the right hand square is a similar factorisation of $\beta \gamma$.

(ii)   The functor $G$ has the structure of a comonad.
The \emph{counit} $\varepsilon$ is the map defined at an object $T$ of $\AA$ by the projection on the factor where $\alpha$ is the map from the corolla $C_{\lambda(T)}$ whose leaves are those of $T$,
$$\Big(\displaystyle{\prod_{\alpha:S \to T} \sbigotimes_v X (T(\alpha)_v) }\Big)^{inv} \to X(T).$$
Indeed, the map $C_{\lambda(T)} \to T$ is unique up to isomorphism, so this projection is well-defined on invariants. 
The \emph{comultiplication} $\delta :G(X) \to G^2(X)$ at $T$ is the map from $G(X)(T)=\Big(\displaystyle{\prod_{\alpha:S \to T} \sbigotimes_v X (T(\alpha)_v) }\Big)^{inv}$ into the object $G^2(X)(T)$ which is 
$$\Big(\displaystyle{\prod_{\alpha:S \to T} \sbigotimes_v \big(\prod_{\beta_v:R_v \to T(\alpha)_v} \sbigotimes_w X ((T(\alpha)_v)(\beta_v)_w)\big)^{inv} }\Big)^{inv}$$
defined as follows.
Take an element $\omega=(\omega_\alpha)$ in $G(X)(T)$, and assume $\omega_\alpha=\sbigotimes_v \omega_\alpha(v)$. (In reality, $\omega_\alpha$ is a sum of such tensors.) Its image under $\delta$ has as its value at $\alpha:S \to T$ and as tensor factor for $v\in S$ an element
$$\delta(\omega)_\alpha (v) \in
\big(\prod_{\beta_v:R_v \to T(\alpha)_v} \sbigotimes_w X ((T(\alpha)_v)(\beta_v)_w)\big)^{inv}.$$
To define this value, note that the maps $\beta_v : R_v \to T(\alpha)_v$ graft together to a factorisation $S \stackrel{\gamma}{\to} R \stackrel{\beta}{\to} T$ of $\alpha$, where $\gamma$ blows up each vertex $v$ to a tree $R_v=R(\gamma)_v$.
Then we set
$$(\delta(\omega)_\alpha(v))_{\beta v}(w):=\omega_\beta(w).$$
Ignoring invariants for simplicity of notation, here is another way to describe the same map $G(X)(T) \to G^2(X)(T)$.
Since the products involved are all finite, they distribute over the tensors and the codomain $G^2(X)(T)$ can be identified with the product
$$\displaystyle{\prod_{S \stackrel{\alpha}{\to} R \stackrel{\beta}{\to} T}} \sbigotimes_w X(T(\beta)_w)$$
where $w$ ranges over vertices of $R$. To compare to the earlier description, each such vertex $w$ lies in $R(\alpha)_v$ for a unique vertex $v$ of $S$. Then $\delta :  G(X)(T) \to G^2(X)(T)$ is the map making the following diagram commute for each pair $S \stackrel{\alpha}{\to} R \stackrel{\beta}{\to} T$:
\begin{eqnarray*}
\xymatrix@M=8pt{
 \displaystyle{\prod_{S \stackrel{\alpha}{\to} T}} \sbigotimes_v X(T(\alpha)_v) \ar@{-->}[rr] \ar[dr]_{proj_\beta} &&  \displaystyle{\prod_{S \stackrel{\alpha}{\to} R \stackrel{\beta}{\to} T}} \sbigotimes_w X(T(\beta)_w) \ar[dl]^{proj_{\alpha,\beta}}
 \\
 & X(T(\beta)_w) &}
  \end{eqnarray*}
 
 (iii) If $X$ is a preoperad, its structure maps can be iterated, so as to give for each morphism $\alpha: S \to T$ in $\AA$ a map 
$$X(T) \to \sbigotimes_{v\in S} X(T(\alpha)_v).$$
This map is well-defined by the associativity constraints on the structure maps $\theta$.
These maps $X(T) \to \sbigotimes_{v\in S} X(T(\alpha)_v)$ for all morphisms $\alpha$ together define a map $\theta: X(T) \to G(X)(T)$, and give $X$ the structure of a coalgebra over the comonad $G$.

We leave it to the reader to check that the maps $\delta$ and $\varepsilon$ indeed satisfy the equations for a comonad, and that a coalgebra structure $\theta:X \to G(X)$ is equivalent to a preoperad structure on $X$. These verifications are a bit lengthy but straightforward.
\end{proof}

 Next, let us consider a preoperad $X$, and the associated category of prealgebras $Prealg(X)$. We claim that this category is again of the form coalgebras for a comonad, as stated in the following proposition.
 An explicit description of the comonad will again be given in the course of the proof.
 
 \begin{prop}\label{Coalg=Comonad}
  Let $X$ be a preoperad. Then the category $Prealg(X)$ of prealgebras over $X$ is the category of coalgebras for a comonad on the category of presheaves on $\RR$,
  $$H_X : Psh(\RR) \to Psh(\RR).$$
 \end{prop}
 
 \begin{proof}
 The proof proceeds along the same lines as that of the previous proposition. The functor $H_X$ is defined for a presheaf $V$ and a tree $T$ in $\RR$ by
 $$H_X(V)(T)= 
 \Big( \displaystyle{\prod_{\alpha: S\to T}} X(S) \otimes V (T/ S)\Big)^{inv}
 =
 \Big( \displaystyle{\prod_{\alpha: S\to T}} X(S) \otimes \sbigotimes_\ell V(T(\alpha)_\ell)\Big)^{inv}
 $$
 where $\alpha$ ranges over maps in $\RR$, $\ell$ over leaves of $S$, and $T(\alpha)_\ell$ is the subtree of $T$ above the image of $\ell$ under $\alpha$, all as before. Let us check a few things.

 (i) $H_X(V)$ is a contravariant functor of $T$. Indeed, given $\beta:T \to T'$ in $\RR$, the map $\beta^* : H_X(V)(T') \to H_X(V)(T)$ is the one making each square of the following form commute, for each $\alpha:S \to T$:
  \begin{eqnarray*}
\xymatrix@M=8pt{
H_X(V)(T') \ar[d] \ar@{-->}[r]^{\beta^*} & H_X(V)(T) \ar[d]
 \\
 X(S) \otimes \sbigotimes_\ell V(T'(\beta \alpha)_\ell)\ar[r]
 &
 X(S) \otimes \sbigotimes_\ell V(T(\alpha)_\ell)
 }
  \end{eqnarray*}
  where the lower map comes from the morphisms $T(\alpha)_\ell \to T'(\beta \alpha)_\ell$ induced by $\beta$. These maps arise by factoring $\alpha$ and $\beta\alpha$ each as a map in $\AA$ followed by an external face
  \begin{eqnarray*}
\xymatrix@M=8pt{
S \ar[r]^\alpha & T \ar[r]^\beta & T'
 \\
 S\ar[r]^\delta \ar@{=}[u] &
 S' \ar[r]^{\delta'} \ar[u]^\gamma &
 S'' \ar[u]^{\gamma'}
 }
  \end{eqnarray*}
with $\gamma, \gamma'$ in $\RR$, and $\delta, \delta'$ in $\AA$.  
  Then $\delta$ and $\delta'$ are isomorphisms on the leaves.
  If $\delta(\ell)=\ell'$ and $\delta'(\ell')=\ell''$ then $\beta$ maps $\gamma(\ell)$ to $\gamma'(\ell')$, so the subtree $T(\gamma)_{\ell'}$ of $T$ with root $\gamma(\ell')$ is mapped by $\beta$ to the subtree $T'(\gamma')_{\ell''}$ of $T'$ with root $\gamma'(\ell'')$. These are exactly the trees $T(\alpha)_\ell$ and $T'(\beta \alpha)_\ell$.
  
  (ii)  $H_X$ has the structure of a comonad: The \emph{counit} $H_X(V) \stackrel{\varepsilon}{\to} V$ at $T$ is the map
  $$\Big( \displaystyle{\prod_{\alpha: S\to T}} X(S) \otimes \sbigotimes_\ell V(T(\alpha)_\ell)\Big)^{inv} \to V(T)  $$
  given by the projection onto the factor where $\alpha:S \to T$ is the inclusion $\eta \to T$ of the root of $T$. The corresponding $T(\alpha)_\ell$ is just $T$ in this case, and $X(\eta)=k$ by convention.
  The \emph{comultiplication} $\delta : H_X(V) \to H_X H_X (V)$ at $T$ is to be a map from $H_X(V)(T)= \Big( \displaystyle{\prod_{\alpha: S\to T}} X(S) \otimes \sbigotimes_\ell V(T(\alpha)_\ell)\Big)^{inv}$
  into the invariants of
  $$\displaystyle{\prod_{\alpha: S\to T}} X(S) \otimes \sbigotimes_\ell \big(\displaystyle{\prod_{\beta_\ell: R_\ell \to T(\alpha)_\ell}} X(R_\ell) \otimes \sbigotimes_k V(T(\alpha)_\ell(\beta_\ell)_k)\big).$$
  Here $\ell$ ranges over the leaves of $S$ and $k$ over those of $R_\ell$. As before, this last object is isomorphic to
  $$\displaystyle{\prod_{ S\stackrel{\alpha}{\to} R \stackrel{\beta}{\to}  T}} X(S) \otimes \sbigotimes_\ell X(R(\alpha)_\ell) \otimes \sbigotimes_k V(T(\beta)_k)$$
  where $k$ now ranges over all leaves of $R$. The map $\delta$ is the one making the diagram below commute, for each $ S\stackrel{\alpha}{\to} R \stackrel{\beta}{\to}  T$
  \begin{eqnarray*}
\xymatrix@M=8pt{
 H_X(V)(T) \ar[r]^\delta \ar[d]^{proj_\beta} &
 H_X H_X(V)(T) \ar[d]^{proj_{\beta \circ \alpha}}\\
 X(R) \otimes \sbigotimes_k V(T(\beta)_k) \ar[r]
 &
X(S) \otimes \sbigotimes_\ell X(R(\alpha)_\ell) \otimes \sbigotimes_k V(T(\beta)_k)
 } 
  \end{eqnarray*}
The lower map here comes from the operad structure of $X$, applied to the grafting of the $R(\alpha)_\ell$ onto the leaves of $S$ which results in the tree $R$.

(iii)
  Finally, the structure of a prealgebra $V$ over $X$ as given by maps $\tau^\alpha : V(T) \to  X(S) \otimes \sbigotimes_\ell V(T(\alpha)_\ell)$ for $\alpha : S\to T$ comes down to the structure of a coalgebra $\tau: V \to H_X(V)$ by simply taking the product of the $\tau^\alpha$'s, giving a map
  $$V(T) \to \Big( \displaystyle{\prod_{\alpha: S\to T}} X(S) \otimes \sbigotimes_\ell V(T(\alpha)_\ell)\Big)^{inv}$$ 
  for each tree $T$.
  
  It remains to check the relevant identities, showing that $\delta$ and $\varepsilon$ indeed define a coalgebra structure on $H_X$, and that $\tau : V \to H_X(V)$
 as just described is a coalgebra structure. These are all straightforward and left to the reader.
\end{proof}

Next, consider a map of preoperads $f : X \to X'$. This map induces a map of comonads $H_f:H_X \to H_{X'}$. By composition, $H_f$ defines a functor on coalgebras, which we denote
$$f_! : Coalg(H_X) \to Coalg(H_{X'})$$
because it can be identified with the functor $f_!: Prealg(X) \to Prealg(X')$ under the correspondence of Proposition~\ref{Coalg=Comonad}.

\begin{prop}\label{BaseChange}
 The functor $f_!:Coalg(H_X) \to Coalg(H_{X'})$ has a right adjoint
 $$f^*: Coalg(H_{X'}) \to Coalg(H_{X}).$$
 For an $H_{X'}$-coalgebra $(M', \theta')$, its value $f^*(M)$ can be computed as the equaliser in $Psh(\RR)$
 \begin{eqnarray*}
\xymatrix@M=8pt{
 f^*(M) \ar[r] & H_X(M')\ar@<-.5ex>[r] \ar@<.5ex>[r] & H_X H_{X'}(M')
 } 
  \end{eqnarray*}
 of the parallel maps $H_X(\theta')$ and $(H_X H_f) \circ \delta$ at $M'$.
\end{prop}

\begin{remark}
 Under the identification of Proposition~\ref{Coalg=Comonad}, this defines a change-of-base functor, again denoted 
 $$f^* : Prealg(X') \to Prealg(X).$$
\end{remark}

\begin{proof}[Proof of Proposition~\ref{BaseChange}]
 This is really part of the general theory of comonads or their dual monads.
 The main point is to show that the equaliser $f^*(M')$ in the statement of the proposition carries the structure of an $H_X$-coalgebra 
 $f^*(M') \to H_X(f^*(M'))$. This in turn follows from the fact that $H_X$ preserves this equaliser. To see this, notice that this equaliser
 is coreflexive, i.e. the parallel maps have a common retraction\\
 $$H \varepsilon' : H_X H_{X'} (M') \to H_X(M')$$
 where $\varepsilon'$ is the counit of $H_{X'}$. Next, one observes that the functor $H_X$ preserves coreflexive equalisers. (Remember we work over a field; the same will be true under sufficient flatness conditions.)
 Indeed, the product over $\alpha:S \to T$, the invariants and the tensor with $X(S)$ involved in the definition of $H_X(V)$ evidently do, and it suffices to show that the functor 
 $$ V \mapsto V (T / S)$$
 preserves coreflexive equalisers. Remember that $V(T / S)$ is a tensor product of copies of $V(T_\ell)$ for trees $T_\ell$, one for each leaf $\ell$ of $S$. The fact that $V \mapsto V (T / S)$ preserves coreflexive equalisers now follows by induction from the following well-known fact. The proof is elementary, but we include it for lack of a reference and convenience of the reader. This then completes the proof of the proposition.
\end{proof}

 \begin{lemma}
  Suppose
\begin{eqnarray*}
\xymatrix@M=8pt{
 A \ar[r]^i & B\ar@<-.5ex>[r]_g \ar@<.5ex>[r]^f & C
 } 
 \quad \quad \text{and} \quad \quad 
\xymatrix@M=8pt{
 D \ar[r]^j & E\ar@<-.5ex>[r]_k \ar@<.5ex>[r]^h & F
 } 
 \end{eqnarray*}
  are two coreflexive equalisers in a monoidal category, with retractions $s$ and $t$ respectively.
  Suppose also that $A,B, E$ and $F$ are flat (tensoring with each of them preserves equalisers).
  Then
\begin{eqnarray*}
\xymatrix@M=8pt{
 A \otimes D \ar[r]^{i \otimes j} 
 & B \otimes E \ar@<-.5ex>[r]_{g \otimes k} \ar@<.5ex>[r]^{f \otimes h} 
 & C\otimes F
 } 
 \end{eqnarray*}
  is again a (coreflexive) equaliser.
 \end{lemma}
 
 \begin{proof}
 Consider the diagram
\begin{eqnarray*}
\xymatrix@M=8pt{
 A \otimes D \ar[r] \ar@{->}[d]
 & A \otimes E \ar@<-.5ex>[r]\ar@<.5ex>[r]  \ar@{>->}[d]
 & A\otimes F  \ar@{->}[d]\\
 B \otimes D \ar@{>->}[r] \ar@<-.5ex>[d]\ar@<.5ex>[d]
 & B \otimes E \ar@<-.5ex>[r]\ar@<.5ex>[r]\ar@<-.5ex>[d]\ar@<.5ex>[d]
 & B\otimes F \ar@<-.5ex>[d]\ar@<.5ex>[d]\\
 C \otimes D \ar[r]
 & C \otimes E \ar@<-.5ex>[r] \ar@<.5ex>[r]
 & C\otimes F.
 }
 \end{eqnarray*}
 First of all, using that $A$ and $F$ are flat, it is easy to see that the top left square is a pullback. Now suppose $p:X\to B \otimes E$ is a map equalising $f \otimes h$ and $g \otimes k$. Then using the common retractions $s$ of $f$ and $g$, and $t$ of $h$ and $k$, we see that $p$ equalises  
$\xymatrix@M=2pt{ B \otimes E \ar@<-.5ex>[r]\ar@<.5ex>[r]
 & B\otimes F}$ as well as 
 $\xymatrix@M=2pt{ B \otimes E \ar@<-.5ex>[r]\ar@<.5ex>[r]
 & C\otimes E}$. Since $B$ and $C$ are assumed flat, $p$ factors through $B \otimes D$ and $A \otimes E$ respectively, and by the pullback property also through $A \otimes D$.
\end{proof}

\begin{remark}
 (a) Let $f : X \to X'$ be a map of preoperads.
 Evidently $f_!: Prealg(X) \to Prealg(X')$ preserves quasi-isomorphisms. Since we work over a field, the description of $f^*$ as an equaliser shows that $f^*$ does as well.
 
 (b) If $f:X \to X'$ is itself a quasi-isomorphism, then $f_!$ and $f^*$ are mutually inverse up to quasi-isomorphism.
 Indeed, as $f_!$ also detects quasi-isomorphisms, it suffices to prove that the counit $f_!f^*(M') \to M'$ is a quasi-isomorphism for each $X'$-prealgebra $M'$. This follows by comparing the defining equaliser for $f^*(M')$ with the standard equaliser defining a cofree resolution of $M'$, as in the following diagram
 \begin{equation*}
\xymatrix@M=8pt{
f^*(M') \ar[r]\ar@{-->}[d] 
&H_X(M')\ar[d]^\wr  \ar@<-.5ex>[r] \ar@<.5ex>[r]
&H_X H_{X'}(M')\ar[d]^\wr
\\
 M' \ar[r] 
 &   H_{X'}(M')  \ar@<-.5ex>[r] \ar@<.5ex>[r]
 &  H_{X'}H_{X'}(M') \\}
\end{equation*}  
 (Note: this is a diagram in the underlying category of presheaves.)
\end{remark}

\section{Some remarks concerning signs}\label{AppendixSigns}

In earlier sections, we have omitted some of the routine verifications that the Koszul convention in the definitions gives the correct signs for all the equations. When done explicitly by evaluating on elements, these verifications can be quite lengthy. It is often easier to give diagrammatic arguments, as we will now illustrate in Section~\ref{SignsStructure}. By way of contrast, we will illustrate explicitly checking the signs by evaluation at elements in Section~\ref{SignsBarCobar}. 

%
%

\subsection{Proof of the properties of the cooperad structure on $\bBar X$}\label{SignsStructure}
Recall from Section~\ref{Section:Operads} that for a preoperad $X$ (with structure map $\theta$),
 the cooperadic structure on $\bBar(sX)$ is given on a decomposition $S= S_1 \circ_a S_2$ by  
 $$\bar \Delta^a = twist \circ (id \otimes \tilde \theta) \circ (drop_a \otimes id)$$
where $\tilde \theta$ is $(s \otimes s) \circ \theta \circ s^{-1}$, which  gives on elements the following formula:
$$\bar \Delta^a (\alpha, e \otimes x)= \sum \pm (\alpha_1, e_1 \otimes x_1) \otimes (\alpha_2, e_2 \otimes x_2)$$
where $e=ae_1e_2$, $\tilde \theta (x)= \sum x_1 \otimes x_2$ and $\pm$ means the associated Koszul sign.
To show that $\bBar(sX)$ with structure maps $\bar \Delta$ is a precooperad, we need to prove naturality and coassociativity of $\bar \Delta$ and also that $\bar \partial$ is a  coderivation with respect to $\bar \Delta$. We will now give a diagrammatic proof of these three facts.

\subsubsection{Naturality}\label{B111}
Consider a morphism of trees $\gamma: S \to S'$ in $\AA$ such that $\gamma= \gamma_1 \circ_a \gamma_2 : S_1 \circ_a S_2 \to S'_1 \circ_a S'_2$.
First note that for a tree $T$ and $\alpha : S \to T$ in $\AA$, there exists a factorisation $\alpha = \beta \circ \gamma$ if and only if there exist two factorisations $\alpha_1 = \beta_1 \circ \gamma_1$ and $\alpha_2 = \beta_2 \circ \gamma_2$. By definition, the structure map $\theta$ is natural, and thus $\tilde \theta$ is also. The twist map and the map $drop_a \otimes id$ do not interact with $\gamma_*$. Therefore it is clear that the following diagram commutes
 \begin{equation*}
\xymatrix@M=8pt{
\bBar(sX)(S_1\circ S_2) \ar[r]^{\gamma_*} \ar[d]^{\bar \Delta^a}
&\bBar(sX)(S'_1\circ S'_2) \ar[d]^{\bar \Delta^a}
\\
\bBar(sX)(S_1) \otimes \bBar(sX)(S_2) \ar[r]^{\gamma_{1*} \otimes \gamma_{2*}}
&\bBar(sX)(S'_1) \otimes \bBar(sX)(S'_2)
}
\end{equation*}  
which exactly gives the naturality of $\bar \Delta$.

\subsubsection{Coassociativity}\label{B112}
There are a priori two cases to consider, the ``sequential'' case where $S$ can be written $(S_1 \circ_a S_2) \circ_b S_3= S_1 \circ_a (S_2 \circ_b S_3)$ and the ``parallel'' case where $S$ can be written $(S_1 \circ_a S_2) \circ_b S_3= (S_1 \circ_b S_3) \circ_a S_2$.
We will only check the first case.
The second case is very similar to the first one, and is again based on the graded anti-coassociativity of $\tilde \theta$ (implied by the given coassociativity of $\theta$).

In the following diagram (for the first case), $tw$ denotes the twist maps (note that the horizontal twist is not the same as the vertical twist), the horizontal exterior composites are $\bar \Delta^a$ and the vertical ones are $\bar \Delta^b$. We have not labelled all the domains to make the diagram readable and have written $\bBar$ instead of $\bBar(sX)$.
\begin{equation*} 
\! \! \! \! \! \! \! \!\! \! \! \!\! \! \! \!\! \! \! \!\! \! \! \!
\xymatrix@M=12pt{
\bBar(S) \ar[d]^{drop_b} \ar[rr]^{drop_a}
&& . \ar[r]^{id \otimes \tilde \theta^a}\ar[d]^{drop_b}
& . \ar[rr]^{tw}\ar[d]^{drop_b}
&& \bBar(S_1) \otimes  \bBar(S_2 \circ_b S_3)\ar[d]^{drop_b}
\\
. \ar[d]^{id \otimes \tilde \theta^b}  \ar[rr]^{drop_a}
&& . \ar[d]_{id \otimes \tilde \theta^b} \ar[r]^{id \otimes \tilde \theta^a} \ar@{}[rd]^{C}
& . \ar[d]^{id \otimes \tilde \theta^b} \ar[rr]^{tw}
&& .  \ar[d]^{id \otimes \tilde \theta^b}
\\
.\ar[rr]^{drop_a}\ar[d]^{tw}
&& . \ar[r]^{id \otimes \tilde \theta^a} \ar[d]^{tw}
& .\ar[rr]^{tw}\ar[d]^{tw}
&& .\ar[d]^{tw}
\\
\bBar(S_1 \circ_a S_2) \otimes  \bBar(S_3)\ar[rr]^{drop_a}
&& . \ar[r]^{id \otimes \tilde \theta^a}
& .\ar[rr]^{tw} 
&& \bBar(S_1) \otimes  \bBar(S_2) \otimes  \bBar(S_3)
}
\end{equation*}  
The square in the middle denoted by $C$ anticommutes because of the  graded anti-coassociativity of $\tilde \theta$.
There are three other small squares which anticommute: the first two of the first column, and the top one of the middle column.
The remaining small squares commute, so that the total outer square commutes. This gives the coassociativity of $\bar \Delta$ for the sequential case. As said, the parallel case is similar.

\subsubsection{Coderivation}\label{B113}
We show separately that $\partial_{ext}$ and $\partial_{int}$ are coderivations with respect to $\bar \Delta$.

For $\partial_{int}$, notice first that $\partial_{sX}$ is an anticoderivation with respect to $\tilde\theta$ as by hypothesis $\partial_X$ is  a coderivation with respect to $\theta$. Using that $drop_a$ anticommutes with $\partial_{sX}$, this implies that $\partial_{int}$ is a coderivation with respect to $\bar \Delta$.

For $\partial_{ext}$, recall that an element $e \otimes x$ is sent to $(-1)^{e+x} \sum_i (-1)^{i-1}(\partial_i e \otimes \partial^*_{e_i} x)$.
Define $P_X$ to be $\bigoplus_{T_1,T_2} k[en_{T_1}] \otimes X(T_2)$. Its  differential is $\partial^P = (id \otimes \partial^*) \circ (\partial \otimes id)$,
where $\partial^*$ is the sum of the $\partial^*_{e_i}$ (without signs) and $\partial$ is $(-1)^{e+x} \sum_i (-1)^{i-1}\partial_i$, both acting on just one side of the tensor product. This definition implies  that $\partial_{ext}$  is the composite $\bBar(sX) \mono P_X \stackrel{\partial^P}{\to} P_X \stackrel{diag}{\to} \bBar(sX)$  (where $diag$ means the projection onto the diagonal $T_1=T_2$ summand).
Similarly, define $\bar \Delta^a$ on $P_X$ for an edge $a$ as 
$$tw \circ (id \otimes \theta^a) \circ (drop_a \otimes id)$$
(with the convention that if $a \notin T_2$,  $\theta^a=0$ and if $a \notin T_1, drop_a=0$),
so that $\bar \Delta^a$ is the composite  $\bBar(sX) \mono P_X \stackrel{\bar \Delta}{\to} {P_X \otimes P_X} \stackrel{diag \otimes diag}{\longrightarrow} \bBar(sX) \otimes \bBar(sX)$.

The proof is then given by the following commutative diagram, where $\partial$ means $\partial \otimes id$, $\partial^*$ means $id \otimes \partial^*$, and for $d$ a map on $P_X$, $d^2$ on $P_X \otimes P_X$ means the usual $d \otimes id + id \otimes d$. The two squares with $(-1)$ inside anticommute, and all the other small squares (and the single triangle) commute.
This proves that the exterior rectangle commutes.
%
%
%

 \begin{equation*}
 \mkern-140mu
\xymatrix@M=7pt{
\bBar(sX) \ar[dd]^{\partial_{ext}} \ar@{>->}[r]
& P_{sX} \ar@{}[rdd]^{(-1)}  \ar[dd]^{\partial_{ext}} \ar[r]^{s^{-1}}
& P_X \ar@{}[rrd]^{(-1)} \ar[dd]^{\partial_{ext}}  \ar[rr]^{drop} \ar[rd]^{\partial}
&& P_X \ar[r]^{tw \circ \theta \ \ } \ar[d]^{ \partial}
& P_X \otimes P_X  \ar[d]^{ \partial^2} \ar[r]^{s \otimes s}
& P_{sX}  \otimes P_{sX} \ar[dd]^{(\partial^P)^2} \ar@{->>}[r]^{diag^{\otimes 2} \ }
& \bBar(sX) \otimes \bBar(sX)  \ar[dd]^{\partial_{ext}^2}
\\
 & & & P_X\ar[dl]^{\partial^*} \ar[r]^{drop} & 
P_X  \ar[r]^{tw \circ \theta \ \ }  \ar[d]^{\partial^*}
& P_X \otimes P_X  \ar[d]^{(\partial^*)^2} 
 \\
\bBar(sX) \ar@{>->}[r]
 & P_{sX}  \ar[r]^{s^{-1}}
 & P_X \ar[rr]^{drop}
 && P_X   \ar[r]^{tw \circ \theta \ \ }
& P_X \otimes P_X\ar[r]^{s \otimes s}
& P_{sX}  \otimes P_{sX} \ar@{->>}[r]^{diag^{\otimes 2} \ }
& \bBar(sX) \otimes \bBar(sX) 
}
\end{equation*}

%
%

\subsection{Translation to the case of algebras}
The operadic case involves structure maps of the form $X(S_1 \circ S_2) \to X(S_1) \otimes X(S_2)$, or more generally $X(T) \to X(S) \otimes X(T/S)$, while the algebraic case involves maps of the form $M(T) \to X(S) \otimes M(T/S)$.
The diagrams involved in the algebraic case are of the same form as in the operadic case, except that the first component  of the tensor product in the codomain is in  $X$ instead of $M$. 
The properties, such as associativity, of an algebra induce the corresponding properties of its bar construction in a very similar way.
This is easier to see with diagrams rather than by writing equations with elements. We shall indicate it for the unshifted bar construction.

More precisely, recall from Section~\ref{defalgebras} that the structure of  a prealgebra $M : \RR^{op} \to Ch $ over a preoperad $X : \AA^{op} \to \Ch$ is given by maps $\theta^\alpha :M(T) \to X(S) \otimes M(T/S)$ for each $\alpha : S\to T$ in $\RR$.
In Section~\ref{Section:Alg}, we have defined coalgebraic structure maps on the (unshifted) bar construction of $M$. 
More explicitly, using the notation from Remark~\ref{Rem64},
we obtain maps $\overline \nabla^\alpha$ (of degree $-1$), given in terms of coinvariants for $y\in M(U)$ and $e \in en^+(U)$ by
$$\overline\nabla^\alpha (e,y)= (-1)^{f \cdot (\Sigma g_\ell + y_\ell)+e+y}
((f,(g_\ell)_\ell), x\otimes \sbigotimes_\ell y_\ell).$$
Note that each such map $\overline\nabla^\alpha$ can be decomposed as $twist \circ (id \otimes \theta) \circ (drop_{r} \otimes id)$
where $r$ denotes the root edge of $U$, and the $twist$ map puts in the right order the enumerations of the trees (inner edges of $S$ and inner edges plus root of the $T_\ell$'s) and the factors $x$ and the $y_\ell$'s.

The naturality of the structure maps $\overline \nabla$ follows directly from the naturality of $\theta$, as in~\ref{B111}. 

The graded coassociativity now has just one case, where $U$ is decomposed into three levels, as $(R \circ (S_\ell)_\ell) \circ (T_{\ell j})_{\ell,j}$. The map $\overline \nabla$ can either first split $U$ just above $R$ and then above each $S_\ell$, or can first split $U$ above the grafting $(R \circ (S_\ell)_\ell)$ and then above $R$. The coassociativity diagram is then of the same kind as the one in~\ref{B112}, except that instead of the edge $b$ there is now the set of the roots of the $S_\ell$'s (and $drop_b$ is thus replaced by dropping all these edges connecting $R$ and the $S_\ell$'s). Then all the squares in the diagram strictly commute, except the top left one which commutes or anticommutes depending on the parity of the number of $S_\ell$'s. This gives the graded coassociativity.

The proof that the differential of $\bBar M$ (as defined in Section~\ref{BarAlgCoinv}) is a graded coderivation with respect to $\overline \nabla$ can again be split into two, one for the internal part and one for external part of the differential.
For the internal part, we use that $\partial_M$ is a coderivation and that the dropping the root anticommutes with $\partial_M$. For the external part, the same method as in~\ref{B113} can be used, with two differences. First, the external differential now has two parts (but this is easily dealt with by defining $\partial^*=\sum (e_i^*-v_i^*)$). Secondly, $\overline \nabla$ sends $P_M$ to $P_X \otimes \sbigoplus_n P_M^{\otimes n}$ (where $n\geq 1$ is the number of subtrees above the cut). The second difference is dealt with by replacing $\partial^2$ (resp. $(\partial^*)^2$) in the diagram by $\sbigoplus \partial^{n+1}$ (resp. $\sbigoplus  (\partial^*)^{n+1}$), the induced differential on the tensor product. The rest of the proof is the same.

\subsection{Some signs in the proofs of the bar-cobar adjunction for (co)operads}\label{SignsBarCobar}

Recall that the bijective correspondence  between morphisms
 $\varphi :\coBar N \to M$ and $\psi : N \to \Bar M$ is given by the formulas 
$$\varphi_S(\alpha,y\otimes e)=(-1)^e\alpha^*\psi_T(y)_{1_T}(e)$$
 $$\psi_S(y)_\alpha(e)=(-1)^e\varphi_T(1_T, \alpha_*(y) \otimes e). $$
 Let us check first that $\varphi$ preserves the differential if and only if $\psi$ does, with signs. Then we will show that one is compatible with the operadic coproduct if and only if the other one is.
 
 By definition, $\psi$ preserving the differential means for $y \in N(S)$ that:
 $$\psi_S(\partial y) = \partial \psi_S(y),$$
 which evaluated in $\alpha:S\to T$ and $e \in en(T)$ gives
 $$\psi_S(\partial y)_\alpha(e)= \partial_M (\psi_S(y)_\alpha(e))
 +(-1)^y \int_d d^* \psi_S(y)_{d\alpha}(de).$$
Rewriting $\psi_S$ in terms of $\varphi_T$, we get
$$(-1)^e \varphi_T(1_T, \alpha_*(\partial y)\otimes e)
= (-1)^e\partial_M \varphi_T(1_T, \alpha_* y\otimes e)
+(-1)^{y+e+1} \int_d d^*\varphi_{T'}(1_{T'}, (d \alpha)_* y\otimes de).$$
 Cancelling $(-1)^e$, moving terms and writing $z=\alpha_ * y$, this can be rewritten as 
 \begin{equation}\tag{B.1}\label{eqphipsi}
\partial_M \varphi_T(1_T, z\otimes e) 
= \varphi_T(1_T,\partial z\otimes e)
+(-1)^{z} \int_d d^*\varphi_{T'}(1_{T'}, d_* z\otimes de).  
 \end{equation}
  By naturality of $\varphi$, we have $d^*\varphi_{T'}(1_{T'}, d_* z\otimes de) = \varphi_T(d,d_*z\otimes de)$, and by definition of the differential, $\partial(1,z\otimes e)=(1, \partial z \otimes e) + (-1)^z \int_d (d,d_* z\otimes de)$.
  So the equation~(\ref{eqphipsi})
  can be rewritten as $\partial_M \varphi_T(1,z\otimes e)= \varphi_T(\partial(1,z\otimes e))$. This shows that $\varphi$ preserves the differential whenever $\psi$ does. The converse holds by the same argument.
  
  
  Let us also check that each of $\varphi$ and $\psi$ (both of degree $- 1$) commutes with the coproduct if the other does.
  Suppose $(\varphi \otimes \varphi) \nabla (y\otimes e)=-\theta_M \varphi(y \otimes e)$ for $y\in N(T)$. We want to prove that $(\psi \otimes \psi)(\theta_N(y))=-\Delta(\psi y)$ for $y\in N(S)$.
It suffices to prove this equality after applying the isomorphism $\mu$ of Section~\ref{barcomplexpreoperad}. For the left hand side, this gives
  \begin{align*}
  & \mu(\psi \otimes \psi)(\theta(y))_{e_1,e_2} \\
  =&  \psi(y_1)_{e_1} \otimes \psi(y_2)_{e_2} (-1)^{y_1+e_1(y_2+1)}  \\
  =&\varphi(\alpha_{1*}y_1 \otimes e_1) \otimes \varphi(\alpha_{2*}y_2 \otimes e_2) (-1)^{e_1+e_2+y_1+e_1(y_2+1)}  \\
  =&(\varphi \otimes \varphi) ((\alpha_{1*}y_1 \otimes e_1) \otimes (\alpha_{2*}y_2 \otimes e_2)) (-1)^{y_1+e_1+e_1+e_2+y_1+e_1(y_2+1)}  \\
  =&(\varphi \otimes \varphi) (\nabla(\alpha_* y \otimes e)) (-1)^{y+e_1y_2}) (-1)^{e_1+e_2+e_1(y_2)}  \\
  =&(\varphi \otimes \varphi) (\nabla(\alpha_* y \otimes e)) (-1)^{y+e_1+e_2}  \\
  =& \theta \varphi(\alpha_* y \otimes e) (-1)^{y+e_1+e_2}   
  \end{align*}
Considering the right hand side of what we want to prove, we have
  \begin{align*}
  & (\mu\Delta(\psi y))_{e_1,e_2} \\
  =&  \Delta'(\psi y)_{e_1,e_2} \\
  =& (-1)^{y+1}\theta(\psi(y)(ae_1e_2))\\
  =& (-1)^{y+1}\theta((-1)^e \varphi_T(1, \alpha_* y \otimes e)) \\
  =& (-1)^{y+1+e}\theta(\varphi_T(1, \alpha_* y \otimes e))
  \end{align*}
  So, using the previous calculation, we obtain $(\mu\Delta(\psi y))_{e_1,e_2} = -(\mu(\psi \otimes \psi)\theta(y))_{e_1,e_2}$, which proves that if $\varphi$ is compatible with the coproduct, than $\psi$ is also. The same computation also proves the converse.


\bibliographystyle{plain}

\bibliography{biblio}

\end{document}